\documentclass[12pt]{amsart}
\usepackage[colorlinks=true, pdfstartview=FitV, linkcolor=blue, citecolor=green, urlcolor=black,filecolor=magenta]{hyperref}

\usepackage{graphicx, verbatim,amsmath,amssymb}

\newcommand{\fff}{{\mathcal F}}

\newcommand{\lcal}{{\mathcal L}}
\newcommand{\ppp}{{\mathcal P}}

\newcommand{\rrr}{{\mathcal R}}

\newcommand{\ilim}{\varprojlim}

\newcommand{\poly}{\Delta}

\newcommand{\omegan}{\Omega^{(n)}}

\newcommand{\omeganm}{\Omega^{(n-1)}}

\newcommand{\ren}{{\mathcal T}}

\newcommand{\mmmn}{\mathcal{M}^{(n)}}
\newcommand{\mmmN}{\mathcal{M}^{(N)}}

\newcommand{\T}{{\bf T}}
\newcommand{\R}{{\mathbb R}}

\newcommand{\N}{{\mathbb N}}

\newcommand{\Z}{{\mathbb Z}}

\newcommand{\vecv}{{\vec{v}}}

\newcommand{\type}{{label}}

\newtheorem{thm}{Theorem}[section]
\newtheorem*{thm*}{Theorem}
\newtheorem{cor}[thm]{Corollary}
\newtheorem{lem}[thm]{Lemma}
\newtheorem{prop}[thm]{Proposition}
\newtheorem{dfn}[thm]{Definition}

\everymath{\displaystyle}

\theoremstyle{definition}
\newtheorem{ex}{Example}

\begin{document}
\title{Fusion tilings with infinite local complexity}
\author{Natalie Priebe Frank and Lorenzo Sadun}
\date{\today}

\address{Natalie Priebe Frank\\Department of Mathematics\\Vassar
  College\\Poughkeepsie, NY 12604} \email{nafrank@vassar.edu}
\address{Lorenzo Sadun\\Department of Mathematics\\The University of
  Texas at Austin\\ Austin, TX 78712} \email{sadun@math.utexas.edu}
\thanks{The work of the second author is partially supported by NSF
  grant DMS-1101326} 

\begin{abstract}
We propose a formalism for tilings with infinite local complexity
(ILC), and especially fusion tilings with ILC.   
We allow an infinite variety of tile types but
  require that the space of possible tile types be compact. Examples
  include solenoids, pinwheel tilings, tilings with fault lines, and
  tilings with infinitely many tile sizes, shapes, or labels. Special attention
is given to tilings where the infinite local complexity comes purely from
geometry (shears) or comes purely from combinatorics (labels). 
We examine spectral properties of the invariant measures and 
define a new notion of complexity that applies to ILC tilings. 
\end{abstract}

\maketitle

\setlength{\baselineskip}{.6cm}


\section{Definitions}

In the standard theory of tiling dynamical systems, which is motivated in
part by the discovery of aperiodic solids, tilings are constructed from a
finite number of tile types that can be thought of as atoms.  It is
usually assumed that these prototiles have only a finite number of
possible types of adjacencies and this is called {\em finite local
  complexity}, or FLC.  Recently, many interesting tiling models have arisen
that do not satisfy this property.  

In a tiling with {\em infinite local complexity} (ILC), 
there are infinitely many 2-tile patterns, that is, infinitely
many ways for two tiles to meet. Consider an ILC tiling that
is made from a finite set of tile types.  If we
then `collar' the tiles by marking each tile by the pattern of tiles around it,
we obtain an equivalent tiling with infinitely many tile types. We therefore
allow arbitrarily many tile types from the start, but require that the
space of tile types should be compact.  In particular, it 
should only be given the (usual) discrete topology when the set of 
possible tiles is finite.

\subsection{Tiles, labels, patches, tilings, and hulls.}

We will construct tilings of Euclidean space $\R^d$.  A {\em tile $t$}
is a pair $t=(supp(t),\type(t))$, where $supp(t)$, the support, is a
closed set homeomorphic to a ball in $\R^d$, and $\type(t)$ is an
element of some compact metric space $\lcal$.  We will move tiles
around using the action of translation by all elements of $\R^d$.  An
element $x \in \R^d$ acts on a tile $t$ by acting on its support: $t-x
= (supp(t)-x, \type(t))$.  Two tiles $t_1$ and $t_2$ are said to be
{\em equivalent} if there is an $x \in \R^d$ such that $t_1-x = t_2$.
In particular, they must have the same label. Equivalent tiles 
are said to have the same {\em tile type}. 
For convenience we collect one representative of each tile type into a
{\em prototile set} $\ppp$. This is equivalent to picking a point in
the tile, called a {\em control point}, that is placed at the origin.
If the tiles are convex, then it is often convenient to pick each
tile's control point to be its center of mass.

If a sequence of tile labels 
converges in $\lcal$ to a limiting label, we require that
the supports of the corresponding prototiles converge in
the Hausdorff metric to the support of the prototile with the limiting label.
This condition constrains both the shapes of the prototiles and the locations
of the control points. 
This is where our paradigm separates from the `usual' tiling theory:
not only can there be infinitely many tile types, but those tile types
may approximate one another. 

The condition also implies that all 
tiles with a given label are equivalent. Labels serve primarily to distinguish
between tile types whose prototiles have identical supports.

Given a fixed prototile set $\ppp$, we can form patches and tilings
using copies of tiles from $\ppp$.  A {\em $\ppp$-patch} (or {\em
  patch}, for short) is a connected union of tiles equivalent to tiles
from $\ppp$, whose supports intersect only on their boundaries.  A
{\em tiling $\T$} is an infinite patch whose support covers all of
$\R^d$.  An equivalent definition of finite local complexity (FLC)
is that $\T$ contains only a finite number of two-tile patches up to
translational equivalence.

We make a metric on tiles that will be extended to tilings.  The
distance between two tiles is $d(t_1,t_2) =
max\{d_H(supp(t_1),supp(t_2)), d_\lcal(\type(t_1),\type(t_2))\}$,
where $d_H$ is Hausdorff distance on subsets of $\R^d$ and $d_\lcal$ is the
metric on $\lcal$.  We can extend this to finite patches that have a
one-to-one correspondence between their tiles by letting the distance
between two patches be the maximum distance between pairs of
corresponding tiles.  Finally we extend this to tilings by saying that
the distance between two tilings is the minimum $\epsilon$ for which
the two tilings have patches containing the ball of radius
$1/\epsilon$ around the origin that differ by at most $\epsilon$.  If
there is no $\epsilon \le 1$ for which this is true, we simply define
the distance between the tilings to be $1$.

Given a fixed tiling $\T$, the {\em continuous hull $\Omega_\T$} of
$\T$ is the closure of the translational orbit of $\T$:
$$\Omega_\T = \overline{\{\T-x \text{ such that } x \in \R^d\}}$$
More generally, a {\em tiling space $\Omega$} is a closed,
translationally invariant set of tilings constructed using some fixed
prototile set.  In both cases we have a dynamical system $(\Omega,
\R^d)$ which we can study from topological or measure-theoretic
perspectives.

The tiling space $\Omega$ is necessarily compact. This follows from
the compactness of $\ppp$ together with the fact that the possible
offsets between the control points of two adjacent tiles lie in a
closed and bounded subset of $\R^d$. Given any sequence of tilings
$\T_n \in \Omega$, there is a subsequence whose tiles at the origin
converge in both label and position, a subsequence of {\em that} for
which the tiles touching the ``seed'' tile converge, a further subsequence
for which a second ring of tiles converges, and so on. From
Cantor diagonalization we then get a subsequence that converges everywhere. 
Note that if $\ppp$ were not  required
to be compact, then $\Omega$ typically would not be compact, either; 
a sequence of tilings with a non-convergent sequence of
tiles at the origin does not have had a convergent 
subsequence.
 
\subsection{Forms of infinite local complexity}

There are a number of standard examples of tilings with ILC, exhibiting 
different ways that ILC emerges. In some tilings, tiles appear in an
infinite number of shapes, or even an infinite number of sizes. In others,
the geometry of the tiles is simple, but tiles come with an infinite number
of labels. In others, there are only finitely many tile types, but 
infinitely many ways for two tiles to meet. 

Bratteli-Vershik systems \cite{SOMETHING} are associated with
1-dimensional subshifts (and hence with FLC tilings) and also with
non-expansive automorphisms of a Cantor set.  These non-expansive
automorphisms can be realized as ILC fusion tilings with an infinite
tile set.

In the pinwheel tiling \cite{Radin.pinwheel}, all tiles are 1--2--$\sqrt{5}$
right triangles, but the triangles point in infinitely many directions,
uniformly distributed on the circle. With respect to translations, that
means infinitely many tile types. Versions of the generalized pinwheel
\cite{Sadun.pin} exhibit tiles that are similar triangles, pointing in 
infinitely many directions and having infinitely many different sizes. 
 
ILC tilings can also occur with 
finite tile sets. One can have shears along a ``fault line'' where the
tiles on one side of the line are offset from the tiles on the other side. 
For instance, imagine a tiling of $\R^2$ by two types of tiles, one a rectangle
of irrational width $\alpha$ and height 1, and the other a unit square. 
Imagine that these tiles assemble into alternating rows of unit-width and
width-$\alpha$ tiles. There will be infinitely many 
offsets between adjacent tiles along each boundary between rows. This behavior
occurs frequently in tilings generated by substitutions and generalized
substitutions when the linear stretching factor
is not a Pisot number \cite{Danzer, Me.Robbie, Me.Lorenzo, Kenyon.rigid, 
Sadun.pin}.

\subsection{Outline}

In Section \ref{sec-fusion} we review the formalism of fusion for FLC 
tilings and adapt the definitions to the ILC setting. Most well-known 
examples of ILC tilings are fusion tilings, and much
can be said about them. Many of the results
of \cite{Fusion} carry over, only with finite-dimensional 
vectors replaced by measures and with matrices replaced by maps of measures. 
In Section \ref{sec-basicdynamics} we consider dynamical properties such as 
minimality and expansivity. Some well-known results concerning expansivity
do {\em not} carry over to the ILC setting, and we define {\em strong
expansivity} to account for the differences. 
In Section \ref{sec-measures} we address the measure theory of ILC tilings,
especially ILC fusion tilings. In Section \ref{sec-complexity} we 
adapt ideas of topological pressure and topological entropy to define a
complexity function for ILC tilings. We also relate this complexity to
expansivity. Finally, in Section \ref{sec-geography} we survey the landscape of ILC
tilings, defining different classes of ILC tilings, and seeing how various
examples fit into the landscape.  In the Appendix 
we present several examples in detail. 

\noindent{\bf Acknowledgments.} We thank Ian Putnam for hospitality and
many helpful discussions. We thank the Banff International Research Station
and the participants of the 2011 Banff workshop on Aperiodic Order. 
The work of L.S.~ is partially supported by NSF grant DMS-1101326. 

\section{Compact fusion rules and fusion tilings}\label{sec-fusion}

In previous work \cite{Fusion} 
we defined fusion rules for tilings with finite local
complexity, and now we adapt this definition to handle compact tiling
spaces. Given two $\ppp$-patches $P_1$ and $P_2$ and two translations
$x_1, x_2 \in \R^d$, if the union $P_1- x_1 \cup P_2 - x_2$ forms a
$\ppp$-patch, and if $P_1-x_1$ and $P_2-x_2$ have no tiles in common,
we call the union of $P_1-x_1$ and $P_2-x_2$ the {\em fusion} of $P_1$ to $P_2$.
Patch fusion is simply a version of concatenation for geometric objects.

Intuitively, a fusion tiling develops according to an atomic model: we
have atoms, and those atoms group themselves into molecules, which
group together into larger and larger structures. 
Let $\ppp_0 = \ppp$ be our prototile set, our ``atoms''.  $\ppp$ 
is labeled by a compact set $\lcal_0$.  The first set
of ``molecules'' they form is a set $\ppp_1$ of finite
$\ppp_0$-patches. To each element of $\ppp_1$ we associate a (distinct) label
from a compact set $\lcal_1$.  We use the
notation $\ppp_1 = \{P_1(c) | c \in \lcal_1\}$, where for each $c \in
\lcal_1$ the patch $P_1(c)$ is a finite fusion of elements of
$\ppp_0$.

Similarly, $\ppp_2$ is a set of finite patches,
indexed by a compact label set $\lcal_2$, that are fusions of the
patches in $\ppp_1$, and we write $\ppp_2 = \{P_2(c)| c \in \lcal_2\}$.
We continue in this fashion,
constructing $\ppp_n$ as a set of finite patches that are
fusions of elements of $\ppp_{n-1}$, labeled by some compact set
$\lcal_n$.  While the elements of $\ppp_n$ are technically
$\ppp$-patches, we can also think of them as $\ppp_k$-patches for any
$k \le n$ by considering the elements of $\ppp_k$ as prototiles. At each
stage we assume that the locations of the patches are chosen such 
that $\ppp_k$ is homeomorphic to $\lcal_k$.  We require consistency between
the metrics on the supertile sets in a way we will describe in section \ref{supertile.metric}.

The elements of $\ppp_n$ are called
{\em $n$-supertiles}.  We collect them together into an atlas
of patches we call our {\em fusion rule}:
$$\rrr= \left\{\ppp_n, n \ge 0 \right\} = \left\{ P_n(c) \,\, | \,\,n \in \N 
  \text{ and } c \in \lcal_n \right\}.
$$ 
We say that a finite patch is
{\em admitted by $\rrr$} if it can be arbitrarily well approximated by
subsets of elements of $\rrr$.  If it actually appears inside $P_n(c)$ for some
$n$ and $c \in \lcal_n$, we say it is {\em literally admitted} and if
it appears only as the limit of literally admitted patches we say it
is {\em admitted in the limit}.  
A tiling $\T$ of $\R^d$ is said to be a {\em fusion tiling with fusion
  rule $\rrr$} if every patch of tiles contained in $\T$ is admitted
by $\rrr$.  We denote by $\Omega_\rrr$ the set of all $\rrr$-fusion
tilings.  

For any $n$ we may consider the related space $\omegan_\rrr$ that consists
of the same tilings as $\Omega_\rrr$, except that the prototiles are elements
of $\ppp_n$ instead of $\ppp_0$. That is, we ignore the lowest $n$ levels
of the hierarchy. 
In a fusion tiling, we can break each $n$-supertile into
$(n-1)$-supertiles using the {\em subdivision map $\sigma_n$}, which
is a map from $\omegan_\rrr$ to $\omeganm_\rrr$.  
It is clear that this subdivision map is always a continuous
surjection, but it may not be an injection.  If for all $n$ it is,
then we call the fusion rule {\em recognizable}.
Recognizability means that there is a unique way to decompose each
tiling as a union of $n$-supertiles.

To avoid trivialities, we assume that each $\ppp_n$ consists only of supertiles that
actually appear in some tiling in $\omegan_\rrr$.  We can always
achieve this by shrinking each set $\ppp_n$, eliminating those
spurious supertiles that do not appear in any tilings. We also assume
that each $\ppp_n$ is non-empty, which is equivalent to $\Omega_\rrr$
being non-empty.

\subsection{Metric on $\lcal_n$ and $\ppp_n$}
\label{supertile.metric}
Each $\lcal_n$ is assumed to be a compact metric space, with a metric 
compatible with the metric defined 
on lower levels of the hierarchy. Specifically,
if two labels in $\lcal_n$ are within $\epsilon$, then 
the prototiles $P, P' \in \ppp_n$ that represent them must 
have constituent $(n-1)$-supertiles that are in
one-to-one correspondence, that differ by no more than $\epsilon$ in
the Hausdorff metric on their supports, and whose labels differ by no
more than $\epsilon$ in $\lcal_{n-1}$.

In most examples we will want the metric on $\lcal_n$ to be induced
directly from the metric on $\lcal_{n-1}$. That is, for two
$n$-supertiles to be considered close if and only if all of their
constituent $(n-1)$-supertiles are close, and (by induction) 
if and only if all of their consitutent tiles are close. However, there
are important examples where this is not the case, where 
two $n$-supertiles with different labels may have identical decompositions 
into $(n-1)$-supertiles. This can occur naturally when $\lcal_n$ contains 
collaring information, and is essential to a variety of collaring schemes.

Besides being compact metric spaces, the sets $\lcal_n$ 
must admit well-defined $\sigma$-algebras of measurable subsets. We will
henceforth assume that these algebras have been specified, and speak freely of
measurable subsets of $\lcal_n$. Since $\ppp_n$ is homeomorphic to $\lcal_n$,
we can also speak of measurable subsets of $\ppp_n$.

\subsection{The transition map} \label{transition.map}
A standard construct in both self-similar tiling and substitution sequence
theory is the transition matrix, whose $(i,j)$ entry
counts how
many tiles of type $i$ are found in a substituted tile of type $j$. 
A similar analysis applies to fusions with FLC \cite{Fusion}, 
where $M_{n,N}(i,j)$ tells how many
$n$-supertiles of type $i$ are found in an $N$-supertile of type $j$.
These matrices satisfy $M_{n,N} = M_{n,m}M_{m,N}$ for each integer $m$ between
$n$ and $N$. Many ergodic properties of a fusion tiling space, such as 
whether it is uniquely ergodic, reduce to properties of these matrices
\cite{Fusion}.

An apparent obstacle for fusion rules on non-FLC spaces is that the
spaces $\lcal_n$ that label $n$-supertiles need not be finite.  
However, since we
require each $n$-supertile to be a {\em finite} fusion of
$(n-1)$-supertiles, we can still define the {\em transition map}
$M_{n,N}: \ppp_n \times \ppp_N \to \Z$ by
\begin{eqnarray*} M_{n,N} (P,Q) &=& \#(P \text{ in } Q) \cr
&:= & \text{ the number of } n \text{-supertiles equivalent to } 
P  \text{ in the } N\text{-supertile } Q.
\end{eqnarray*}
Thus the $Q$th `column' $M_{n,N}(*,Q)$ 
gives the breakdown of $Q$ in terms of the
$n$-supertiles that it contains, and will consist of 0's except in
finitely many places.  If there is more than one way that the
$n$-supertiles can be fused to create $Q$ (that is, if the fusion
is not recognizable), we fix a preferred one to use in this and all
other computations.

We will use the transition map in three different ways throughout this
paper: as defined above, as a measure on $\ppp_n$, and as an operator
mapping (``pushing forward'') measures on $\ppp_N$ to
measures on $\ppp_n$.  We give the details on these
three views in subsection \ref{transition.views}.  As in the FLC case,
the transition map determines quite a bit about the invariant
probability measures on $\Omega_\rrr$.

\begin{ex} \label{pinwheel_ex2} {\em Pinwheel tilings.}
  In the pinwheel tiling, all tiles are $1, 2, \sqrt{5}$ right
  triangles, but tilings consist of tiles pointing in infinitely many
  directions. We call a triangle with vertices at $(-1.5,-.5)$,
  $(.5,-.5)$ and $(.5,.5)$ {\em right-handed} and give it label (R,0),
  and a triangle with vertices at $(-1.5,.5)$, $(.5,.5)$ and
  $(.5,-.5)$ is called {\em left-handed} and has label (L,0).  Our
  label set $\lcal_0$ consists of two circles, and the prototile with label
  (R,$\theta$) (resp.(L,$\theta$)) is obtained by rotating the (R,0)
  (resp. (L,0)) prototile counterclockwise around the origin by
  $\theta$.  Two tiles are close in our tile metric if they have the
  same handedness, if their angles $\theta$ are close, and if their
  centers are close. This is the same as being close in the Hausdorff
  metric. 

Likewise, $\lcal_n$ consists of two circles, with
the $(R,\theta)$ $n$-supertile being an expansion by $5^{n/2}$ of the 
$(R,\theta)$ tile, and likewise for the $(L,\theta)$ $n$-supertile. Let 
$\alpha=\tan^{-1}(1/2)$. Each $(R,\theta)$ $n$-supertile is built from
five $(n-1)$-supertiles, two of type $(L,\theta+\alpha)$, one of type
$(L,\theta+\alpha + \frac{\pi}{2})$, one of type $(R,\theta+\alpha)$ and
one of type $(R, \theta+\alpha+\pi)$, arranged as in Figure \ref{pinfusion}.
\begin{figure}[ht]
\includegraphics[width=1.5in]{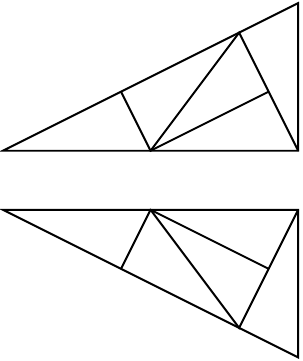}
\caption{The pinwheel fusion rule builds $n$-supertiles from five 
$(n-1)$-supertiles}
\label{pinfusion}
\end{figure}
Likewise, each $(L,\theta)$ $n$-supertile is built from $(n-1)$-supertiles
of type $(R,\theta-\alpha)$,
$(R,\theta-\alpha -\frac{\pi}{2})$, $(L,\theta-\alpha)$ and $(L,\theta-\alpha
+\pi)$. 

To compute the transition map, consider $M_{n, (n+1)}( (H, \omega),
(R, \theta))$.  It equals 1 if $(H,\omega) = (R, \theta + \alpha), (R,
\theta + \alpha + \pi)$, or $(L, \theta + \alpha + \pi/2)$, it equals
2 when $(H,\omega) = (L, \theta + \alpha)$, and it equals 0 otherwise.
Transition for the $(n+1)$-supertile of type $(L, \theta)$ is similar.

Since this example involves infinite many tile orientations, it necessarily
involves infinitely many tile labels. The infinite local complexity has 
both a combinatorial and a geometrical aspect, but both are consequences 
of rotational symmetry. We call this {\em rotational infinite local
complexity.}
\end{ex}

\begin{ex} \label{Non-Pisot_DPV} {\em Shear infinite local complexity.}  
In our next example, each $\ppp_n$ consists of only
four supertiles, and the infinite local complexity comes from the geometry
of how two tiles can touch.
Let $\ppp_0 =
  \left\{\raisebox{-.4cm}{\includegraphics[width=1.0cm]{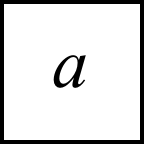}},
    \raisebox{-.4cm}{\includegraphics[width=1.0cm]{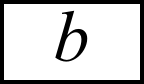}},
    \raisebox{-.4cm}{\includegraphics[width=.6cm]{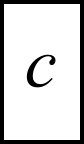}},
    \raisebox{-.4cm}{\includegraphics[width=.6cm]{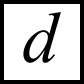}}\right\}$,
  where the long edges are of some fixed length $\alpha$ and the short
  edges are of length $1$.

For the $1$-supertiles we choose $\ppp_1 =
  \left\{\raisebox{-.8cm}{\includegraphics[width=2.0cm]{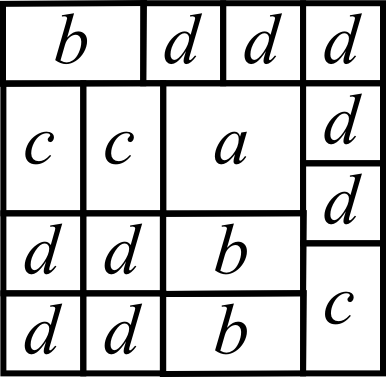}},
    \raisebox{-.8cm}{\includegraphics[width=2.0cm]{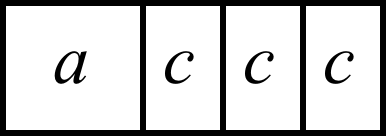}},
    \raisebox{-.8cm}{\includegraphics[width=.8cm]{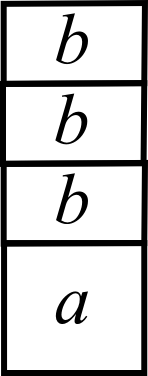}},
    \raisebox{-.8cm}{\includegraphics[width=.7cm]{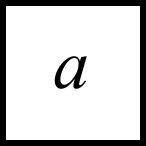}}
  \right\}.$ It is convenient to think of the four supertiles as being
  of types $a, b, c, d$, using the notation $\ppp_1 = \{P_1(a),
  P_1(b), P_1(c), P_1(d)\}$.  We construct the $n$-supertiles from the
  level $(n-1)$-supertiles using the same combinatorics as we did to
  make the 1-supertiles from the prototiles.  For instance, $\ppp_2(a)
  =\raisebox{-2.5cm}{\includegraphics[width=5cm]{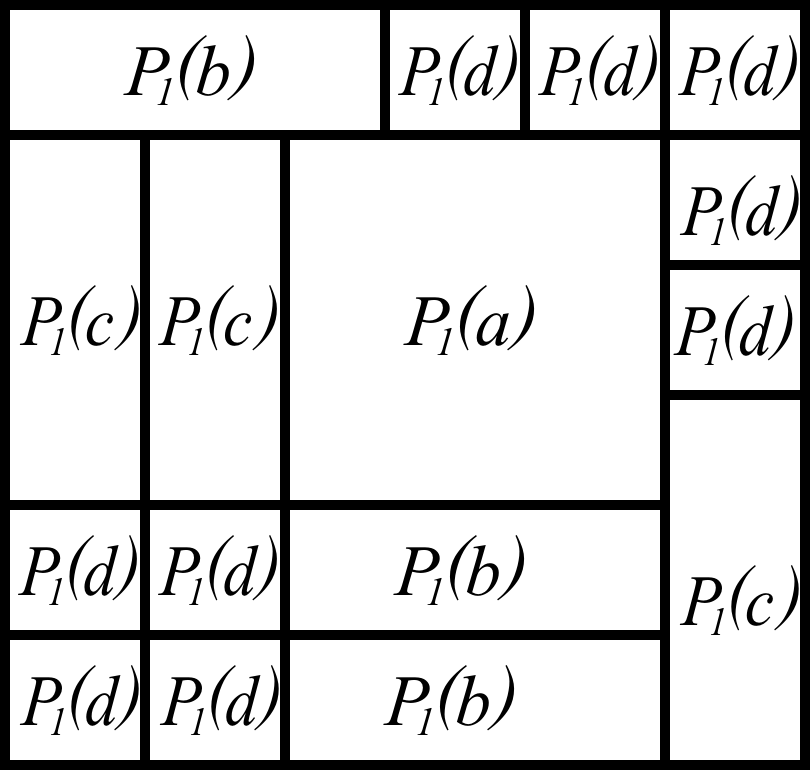}}=
  \raisebox{-2.5cm}{\includegraphics[width=5cm]{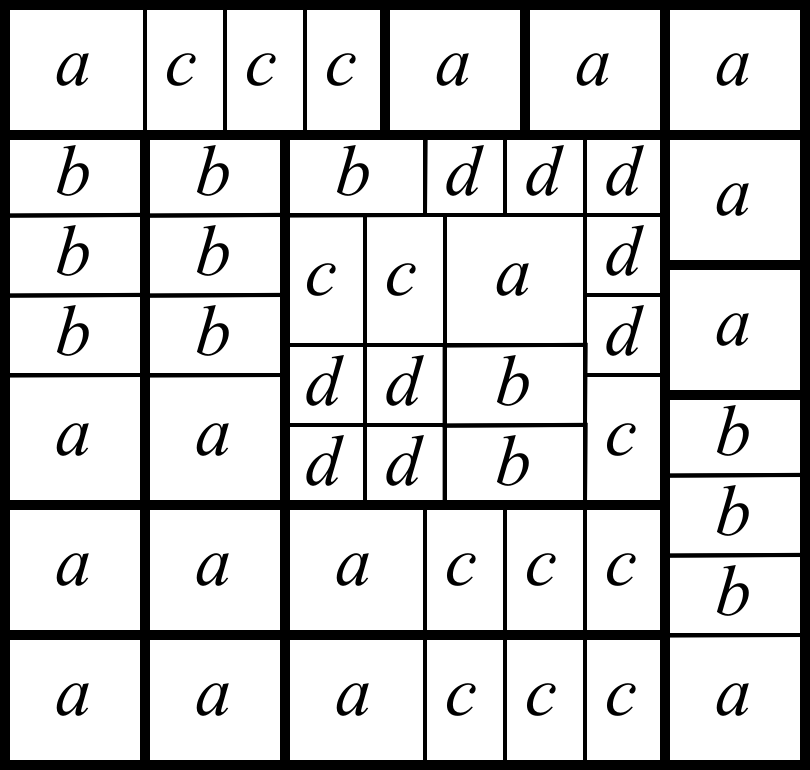}}$.

If $\alpha$ is chosen irrational, then fault lines develop. There
are countably many 2-tile patterns that are literally admitted and 
uncountably many that are admitted in the limit. If $\alpha$ is chosen
rational, then there are only finitely many ways for two $n$-supertiles
to meet, but this number increases with $n$. Either way, 
the large-scale structure of the tilings is different from that of a 
self-similar FLC substitution tiling. These differences show up in the 
spectral theory, cohomology, and complexity.  (See
\cite{Me.Robbie, Me.Lorenzo} and references therein.)

Since at each stage there are only four supertile types, the
transition operator $M_{n,n+1}$ is the matrix $\left[\begin{matrix} 1
    & 1 & 1 & 1\\ 3 & 0 & 3 & 0\\3 & 3 & 0 & 0\\9 & 0 & 0 &
    0 \end{matrix} \right]$ and $M_{n,N}$ is the $(N-n)$th power of this
matrix. 
\end{ex}

\begin{ex}\label{sol-ex}{\em Combinatorial infinite local complexity.}  
Because its (translational)
dynamical system is not expansive, the dyadic solenoid system is not
topologically conjugate to a tiling system with finite local
complexity.  However, it can be expressed as an ILC fusion tiling with
infinitely many tile labels. In this example the geometry is trivial
and the infinite local complexity is purely combinatorial.  

The prototiles are unit length tiles that carry labels $\{A_0, A_1,
\ldots \} \cup \{A_\infty\}$, such that $A_\infty = \lim_{n \to \infty}
  A_n$ is the only accumulation point of the label set.  We define
$$\ppp_1 = \{A_1 A_0, A_2 A_0, A_3 A_0, ...,  A_\infty A_0\},$$
in other words $P_1(k)=A_kA_0$ for $k=1,2,\ldots,\infty$.  
Similarly we define
$$\ppp_2 = \{P_1(2)P_1(1), P_1(3) P_1(1), ..., P_1(\infty) P_1(1)\}.$$ 
In general, every element of $\ppp_{n+1}$ takes the form $P_{n+1}(k) =
P_n(k) P_n(n)$, for all $k \ge n+1$ including $k = \infty$.

Tilings admitted by this fusion rule have $A_0$ in every other spot,
$A_1$ in every fourth spot, and so on, with each species
$A_k$ with $k<\infty$ appearing in every $(2^{k+1})$-st spot. 
In addition, there may be one (and only one) copy of $A_\infty$. 
Since all $A_n$ tiles with $n\ge k$ have the same location (mod $2^{k}$),
the location of an arbitrary $A_n$ with $n \ge k$ gives a map to 
$S^1=\R/2^k\Z$. Taken together, these maps associate a tiling with
a point in the dyadic solenoid $\ilim (\R/2^k\Z)$.
A discrete version of this construction, mimicking an odometer rather
than a solenoid, is called a ``Toeplitz flow''. \cite{Toeplitz}

$M_{n, n+1}(k, l)$ equals $1$ if $k = n$ or if
$k = l$, and otherwise it is 0.  For $N > n$ and relevant values of
$(k,l)$, $M_{n,N}(k,l)$ equals $2^{N-k-1}$ if $k < N$, $1$ if $k
= l$, and 0 otherwise.
\end{ex}

The complexity and ergodic theory of these examples will be worked out in
the Appendix. 

\subsection{Primitivity and the van Hove property}
\label{VanHove-sec}

A fusion rule is said to be {\em primitive} if for any positive
integer $n$ and any open set $U$ of supertiles in $\ppp_n$, there is
an $N > n$ such that every element of $\ppp_N$ contains an element of
$U$.  Primitivity means that the space $\Omega_\rrr$ is fairly
homogeneous, in that each tiling contains patches arbitrarily close to
any particular admissible patch.

A {\em van Hove sequence} $\{A_m\}$ of subsets of $\R^d$ consists of sets
whose boundaries are increasingly trivial relative to their interiors
in a precise sense. 
For any set $A \in \R^d$ and $r > 0$, let
$$ A^{+r} = \{x\in \R^d : \text{ dist}(x, A) \le r\}, $$
where ``dist'' denotes Euclidean distance. A sequence of sets $\{A_n\}$ of sets
in $\R^d$ is called a van Hove sequence if for any $r \ge 0$
$$\lim_{ n \to \infty} \frac{\text{Vol}\left((\partial A_n)^{+r}\right)}
{\text{Vol}(A_n)} 
= 0,$$ 
where $\partial A$ is the boundary of $A$ and $\text{Vol}$ is Euclidean
volume. 

Given a fusion rule $\rrr$, we may make a sequence of sets in $\R^d$
by taking one $n$-supertile for each $n$ and calling its support
$A_n$.  We say $\rrr$ is a {\em van Hove fusion rule} if every such
sequence is a van Hove sequence. Equivalently, a fusion rule is van Hove
if for each $\epsilon >0$ and each $r>0$ 
there exists an integer $n_0$ such that each 
$n$-supertile $A$, with $n \ge n_0$, has $\text{Vol}(\partial A)^{+r}<
\epsilon \text{Vol}(A)$.

\section{Dynamics of ILC tiling spaces}
\label{sec-basicdynamics}

A compact tiling space $\Omega_\rrr$, by definition, is invariant
under translation by elements of $\R^d$. The action of translation is
continuous in the tiling metric and gives rise to a topological
dynamical system $(\Omega_\rrr,\R^d)$.  Tiling dynamical systems have
been studied extensively in the FLC case, and in this section we
investigate the dynamics of ILC tiling systems in general and in the
fusion situation.  We show under what circumstances 
a fusion tiling dynamical system is
minimal, and we discuss what expansivity means and introduce the
related concept of strong expansivity.

\subsection{Minimality}
\label{sec-minimal}
 Recall that a topological dynamical system is said to be
{\em minimal} if every orbit is dense.  

\begin{prop}
If the fusion rule  $\rrr$ is primitive, then
the fusion tiling space 
  $(\Omega_\rrr, \R^d)$ is minimal. Conversely, if $\rrr$ is 
recognizable and van Hove
but is not primitive, then $(\Omega_\rrr, \R^d)$ is not minimal.
\end{prop}
\begin{proof}
First suppose that $\rrr$ is primitive. 
  Fix any $\T \in \Omega_\rrr$ and let $\T'$ be any other element of
  $\Omega_\rrr$.  We will show that for any $\epsilon > 0$ there is a
  $\vecv$ such that $d(\T-\vecv, \T') < \epsilon$.  Let $P' \subset
  \T'$ be the patch of tiles in $\T'$ containing the ball of radius
  $1/\epsilon$ about the origin.  We know that $P'$ is admitted by
  $\rrr$, so there is an $n$ and an $P_n(c) \in \ppp_n$ for which the
  distance between $P'$ and a subpatch of $P_n(c)$ is less than
  $\epsilon/2$.

  Consider an open subset $I \subset \ppp_n$ that contains $P_n(c)$
  and is less than $\epsilon/2$ in diameter.  By
  primitivity, there is an $N > n$ such that every element of $\ppp_N$
  contains an element of $I$.  Choose any $N$-supertile in $\T$ and
  let $\vecv$ be the translation that brings that $N$-supertile to the
  origin in such a way that $I \cap (\T - \vecv) \neq \emptyset$.  That
  is, $\T-\vecv$ has a patch at the origin that is within $\epsilon/2$
  of $P'$.  This means that $d(\T-\vecv, \T') < \epsilon$, as desired.

Now suppose that $\rrr$ is recognizable and 
van Hove but not primitive. Pick an $n$-supertile
$P$ and a neighborhood $U \in \ppp_n$ of $P$  such that there exist 
supertiles of arbitrarily high order that do not contain any elements of 
$U$. By recognizability, this means that there is an open set $\tilde U$
of patches such that there are supertiles of arbitrarily high
order that do not contain patches equivalent to elements of $\tilde U$. 
(The difference between $U$ and $\tilde U$ is that elements of $U$ may
be marked with labels that carry additional information, as with 
collaring, while elements of $\tilde U$ are not.) Let $\T_0$ be a tiling 
featuring an element of $\tilde U$ near the origin. For each sufficiently
large $N$, let $\T_N$ be a tiling where a ball of radius $N$ around the
origin sits in a supertile that does not contain any elements of $\tilde U$.
The existence of such supertiles is guaranteed by the van Hove property.
By compactness, the sequence $\T_N$ has a subsequence that converges to a 
tiling $\T_\infty$, and none of the patches of $\tilde U$ appear anywhere
in $\T_\infty$.  This means that $\T_0$ is not in the orbit closure of 
$\T_\infty$, and hence that $\Omega_\rrr$ is not minimal. 
\end{proof}

If a fusion rule is neither primitive nor van Hove, then the tiling space
may or may not be minimal. For instance, the Chacon substitution $a \to aaba$,
$b \to b$ yields a minimal tiling space, but the substitution $a \to aaba$,
$b \to b$, $c \to ccdc$, $d \to d$ does not.

\subsection{Expansivity and transversals}

A tiling space is said to have {\em expansive} translational dynamics if
there is an $\epsilon>0$ such that the condition $d(\T-x,\T'-x)<\epsilon$
for all $x$ implies that $\T'=\T-y$ for some $y\in \R^d$ with $|y|<\epsilon$.
In other words, in an expansive system, 
tilings that are close must have translates that are no
longer close, unless they were small translates of each other to begin 
with.  

Every tiling in a tiling space is a bounded translate of a tiling that has a control
point at the origin (and therefore contains a tile in $\ppp_0$).  These tilings comprise the  {\em 
transversal } $\Xi$ of a tiling space.
There is a neighborhood of any tiling $\T$ in a tiling space that is homeomorphic to the product of a disk in 
$\R^d$ (for instance, an open set containing the origin in its tile in $\T$) and
a neighborhood in $\Xi$ (for instance, all tilings that are equivalent to $\T$ in a ball of some radius around the origin).   We can think of the transversal as a sort of global Poincar\'e section for the action of translation.

When a tiling has finite local complexity, the transversal is totally
disconnected and in most standard examples is in fact a Cantor set.
Tiling spaces with infinite local complexity can also have totally
disconnected transversals.  This property is invariant under
homeomorphism.

\begin{lem}\label{transversal_topology_invariant}
If two tiling spaces are homeomorphic and one has a totally disconnected 
transversal, then so does the other. 
\end{lem}

\begin{proof}
  Let $\Xi$ and $\Xi'$ be transversals for homeomorphic tiling spaces
  $\Omega$ and $\Omega'$, and let $\phi: \Omega \to \Omega'$ be the
  homeomorphism. Suppose $\T \in \Xi$, with $\Xi$ totally
  disconnected.  We construct a neighborhood of $\T$ in $\Omega$ by
  taking the product of an $\epsilon$-disc in $\R^d$ with a
  neighborhood of $\T$ in $\Xi$.  In other words, a neighborhood of
  $\T$ in $\Xi$ parametrizes the path components of a neighborhood of
  $\T$ in $\Omega$.  To align the control points, we adjust 
the homeomorphism by a translation, so that $\phi(\T) \in \Xi'$. 
Since homeomorphisms
  preserve local path components, there is a neighborhood of
  $\phi(\T)$ in $\Xi'$ that is homeomorphic to a neighborhood of $\T$
  in $\Xi$. Since $\Xi$ and $\Xi'$ are locally homeomorphic, and since
  total disconnectivity is a local property, $\Xi'$ is totally
  disconnected.
\end{proof}

In addition to their transversals being totally disconnected, FLC tiling spaces always have expansive translational dynamics. The converse does not hold.

\begin{thm}\label{can-expand} 
There exists a tiling space $\Omega$ with totally disconnected
transversal and expansive translational dynamics that is not homeomorphic 
to an FLC tiling space. 
\end{thm}

\begin{proof}
The dyadic solenoid has a totally disconnected transversal (namely
an odometer) and is known
not to be homeomorphic to any FLC tiling space, due to its lack of 
asymptotic composants. We will construct tilings that are combinatorially
the same as those of the dyadic solenoid, but in which the tiles do
not all have unit length. By choosing
the tile lengths appropriately, we can make the translational
dynamics expansive. Changing all the tile sizes to 1 while preserving the
position of the origin within a tile gives a homeomorphism 
from this tiling space to the dyadic solenoid, showing that this tiling
space cannot be homeomorphic to an FLC tiling space. 

In this example, our label set is a union of subsets of intervals,
one such subset for each non-negative integer, 
plus a single limit point. The $n$-th
interval describes the possible lengths of tiles of general type $n$,
which we denote $A_n$. The possible sequences of general tile types is
exactly as with the dyadic solenoid, and the lengths of the tiles are
determined from the sequence. 

For each tile $t$ in such a sequence, let $S_t$ be the set of integers 
$n > 10$  such that there is an $A_n$ tile within $n^2$ tiles of $t$.  
The length of $t$ is defined to be $1 + \sum_{n \in S_t} n^{-1}$.  
If a tile is within $n^2$ tiles of
an $A_n$, then it must be at least $2^{n}-n^2$ tiles away from any $A_m$
with $m>n$. For this to be within $m^2$, we must have $m$ of order
$2^{n/2}$. The expression $1 + \sum_{n \in S_t} n^{-1}$  thus converges quickly,
and each tile has a well-defined length. 
A limiting tile $A_\infty$ must have length exactly 1.

Note that the stretching changes the lengths of each $A_n$ tile by
less than $2^{1-\frac{n}{2}}$. However, it stretches out the $2n^2$ tiles
surrounding the $A_n$ tile by a total amount between $2n$ and $2n+n^2
2^{2-\frac{n}{2}}$. If two tilings $\T$ and $\T'$ are very close, with $\T$
having an $A_n$ tile at a location $y$ and $\T'$ having an $A_m$ tile
at the same location (up to a tiny translation), with $m>n$ or
$m=\infty$, then for some $x<(n+1)^2$, the tiles of $\T$ and $\T'$
near $x+y$ are offset by more than 1/4.  Therefore, the tiling
dynamics are expansive, while the transversal remains a totally
disconnected odometer.
\end{proof}

\subsection{Strong Expansivity}

In the proof of Theorem \ref{can-expand}, we introduced expansivity in the
tiling flow via a small change to the sizes of some tile types, changes that
accumulated to give macroscopic offsets. However, this stretching does
not change the dynamics of the first return map on the transversal, which
remains addition by 1 on an odometer. In particular, the action of the
first return map is not expansive, and in fact is equicontinuous. In general,
it is the first return map, and not the tiling flow itself, that determines 
the homeomorphism type of a 1-dimensional tiling space. 

\begin{lem}\label{1D_cannot_expand}
If $\Omega$ is a 1-dimensional tiling space whose canonical
transversal is totally
disconnected, and if the first return map on the transversal is expansive, 
then $\Omega$ is homeomorphic to an FLC tiling space.
\end{lem}

\begin{proof} Suppose that the (iterated) first return map eventually 
separates any two distinct points in the transversal by a distance of 
$\epsilon$ or more. Partition the transversal into a finite number of
clopen sets of diameter
less than $\epsilon$, and associate each clopen set with a tile type. 
Make each tile have length 1. For each point in the transversal we 
associate a tiling, with a tile centered at the origin, and 
with the $n$-th tile marking which clopen set the $n$-th return of the point is
in. That is, the transversal is isomorphic to a (bi-infinite) 
subshift on a finite number of symbols, with the first return map 
corresponding to a shift by one. This extends to a homeomorphism between
$\Omega$ and an FLC tiling space. (Note that this homeomorphism need
not be a topological conjugacy. It commutes with translations if and only
if every tile in the $\Omega$ system has length exactly 1.) 
\end{proof}

For analogous results in higher dimensions, we need a property that
generalizes the expansivity of the first return map.

\begin{dfn} A tiling space $\Omega$ (or any space with an $\R^d$ action)
is called {\em strongly expansive} if there
exists an $\epsilon>0$ such that, for any $\T, \T' \in \Omega$, 
if there is a homeomorphism $h$ of $\R^d$ with $h(0)= 0$ and 
$d(\T - x, \T'-h(x)) < \epsilon$ for all $x\in \R^d$ , then $\T = \T'-x_0$
for some $x_0\in \R^d$ with $|x_0|<\epsilon$. 
\end{dfn}
In other words, the flow separates points that are not already small 
translates of one another, even if you allow a time change between how the
flow acts on $\T$ and how it acts on $\T'$. 

The stretched solenoid of theorem \ref{can-expand} was expansive but not
strongly expansive. This distinction between regular and strong expansivity
is essential for ILC tiling spaces, but unnecessary for FLC tiling
spaces.

\begin{thm}\label{same_for_FLC} 
Strong expansivity implies expansivity. For FLC tiling spaces,
expansivity implies strong expansivity. 
\end{thm}

\begin{proof} Expansivity is a special case of strong expansivity, where
we restrict the homeomorphism $h$ to be the identity. For the partial 
converse, suppose that $\Omega$ is an FLC tiling space with an expansive 
flow with
constant $\epsilon$. Without loss of generality, we can assume that
$\epsilon$ is much smaller than the size of any tile, and is smaller
than the distance between any distinct connected 2-tile patches. 
We will show that $\Omega$ is strongly expansive with constant $\epsilon/3$. 

Suppose that we have tilings $\T$ and $\T'$ and a self-homeomorphism $h$ 
of $\R^d$ such that $h(0)=0$ and $d(\T-x, \T'-h(x)) < \epsilon/3$ for all
$x$. Then we claim that $|x-h(x)| < 2\epsilon/3$ for all $x$, and hence
that $d(\T-x, \T'-x) < \epsilon/3 + 2\epsilon/3=\epsilon$, implying
that $\T'$ is a translate of $\T$ by less than $\epsilon$. 
Since the tiling metric on any small
piece of a translational orbit is the same as the Euclidean metric, and
since $d(\T,\T')<\epsilon/3$, $\T'$ is a translate by less than $\epsilon/3$.
 
To prove the claim, suppose that there is an $x$ with $|x|<3/\epsilon$
such that $|h(x)-x|\ge 2\epsilon/3$. By continuity, we can find such
an $x$ with $|h(x)-x| = 2\epsilon/3$. However, since
$d(\T,\T')<\epsilon/3$, the pattern of tiles in $\T$ and $\T'$ is
exactly the same out to distance $3/\epsilon$ (up to an overall
translation of less than $\epsilon/3$). Since the local neighborhoods
of $\T-x$ and $\T'-x$ agree to within an $\epsilon/3$ translation,
and since $\T'-x$ and $\T'-h(x)$ disagree by at least $2\epsilon/3$,
the tiles near the origin of $\T-x$ and $\T'-h(x)$ are offset by at
least $\epsilon/3$ (and by much less than the spacing between tiles) 
which is a contradiction. 
This is the base case of an inductive argument.  

Next, assuming there are no points with $|x|<3k/\epsilon$ where
$|x-h(x)|\ge 2 \epsilon/3$, we show that there are no points with
$|x|<3(k+1)/\epsilon$ with $|x-h(x)|\ge 2 \epsilon/3$. If such an $x$
exists, take $y=kx/(k+1)$, so $|y|<3k/\epsilon$ and $|x-y|<3/\epsilon$.
Since $d(\T-y, \T'-h(y))<\epsilon/3$ and $|y-h(y)|$
is small, the pattern of tiles in $\T$ and $\T'$ are the same out to
distance $3/\epsilon$ from $y$. Repeating the argument of the previous
paragraph shows that $x$ cannot exist.
\end{proof}

\begin{thm} \label{strong_invariant}
If $\Omega$ is a strongly expansive tiling space and $\Omega'$
is a tiling space homeomorphic to $\Omega$, with the homeomorphism
sending translational orbits to translational orbits,
then $\Omega'$ is strongly expansive. 
\end{thm}

\begin{proof} Suppose that $\Omega$ is strongly expansive with
constant $\epsilon$, that $\Omega'$ is not strongly expansive, and 
suppose that $f: \Omega' \to \Omega$ is a homeomorphism that preserves
orbits. Since $f$ is uniformly continuous, there is a $\delta$ such that
any two points within $\delta$ in $\Omega'$ are mapped to points within
$\epsilon$ in $\Omega$. We can find a pair of tilings $\T'_1, \T'_2 \in
\Omega'$, not small translates of one another, 
and an $h': \R^d \to \R^d$ such that $d(\T'_1-x, \T'_2-h'(x))<\delta$
for all $x$. Let $\T_1=f(\T'_1)$ and $\T_2=f(\T'_2)$. 
Since $f$ maps orbits to orbits, there is a homeomorphism
$\gamma_{1,2}: \R^d \to \R^d$ such that $f(\T'_1-x)=\T_1-\gamma_1(x)$ and 
$f(\T'_2-x)=\T_2-\gamma_2(x)$. Note that $\gamma_{1,2}(0)=0$. 
Since $d(\T'_1-x,\T'_2-h'(x))<\delta$, 
$d(\T_1-\gamma_1(x), \T_2-\gamma_2\circ h'(x))<\epsilon$. Taking 
$y=\gamma_1(x)$
and $h(y)=\gamma_2\circ h' \circ \gamma_1^{-1}$, we have that 
$d(\T_1-y, \T_2-h(y))<\epsilon$, which implies that $\T_1$ and $\T_2$ are small
translates of one another, which implies that $\T_1'$ and $\T_2'$ are small
translates of one another, which is a contradiction. 
\end{proof}

\begin{lem}\label{strong-expansive-first-return}
A 1-dimensional tiling space is strongly expansive if and 
only if the first return map on the 
transversal is expansive.
\end{lem}

\begin{proof} If the first return map is not expansive, then one can find
arbitrarily close tilings $\T, \T'$ whose orbits on the first return map
remain close. The type of the $n$th tile of $\T$ encountered under 
translation must therefore be close
to the tile type of $\T'$, and so the length of each tile in $\T$ must
be close to the length of the corresponding tile in $\T'$. By taking $h(x)$
to increase by the length of a tile in $\T'$ when $x$ increases by the length
of a corresponding tile in $\T$, and by keeping the derivative 
$h'(x)$ constant on each such interval, we ensure that $\T-x$ remains close
to $\T'-h(x)$ for all $x$. 

Conversely, if the first return map is expansive, then any two tilings 
eventually have substantially different sequences of tiles. If $\T-x$
and $\T'-h(x)$ remain close, then, as $x$ increases, the number of vertices
that cross the origin will be the same for $\T$ and $\T'$, since if there
is a vertex at the origin in $\T-x$, then there must be a vertex within 
$\epsilon$ of the origin in $\T'-h(x)$ (and vice-versa). Thus the $n$th tile
of $\T$ must line up with the $n$th tile of $\T'$. But these are eventually
different by more than $\epsilon$.
\end{proof}

\begin{cor}A 1-dimensional tiling space is homeomorphic to a 1-dimensional
FLC tiling space if and only if it is strongly expansive and has 
totally disconnected transversal.
\end{cor}
 
 \begin{proof}
   Suppose that $\Omega$ is a one-dimensional tiling space
   homeomorphic to an FLC tiling space $\Omega'$.  Then since
   $\Omega'$ has a totally disconnected transversal, so must $\Omega$
   by Lemma \ref{transversal_topology_invariant}.  Moreover, since the
   homeomorphism must preserve path components and those are the
   translational orbits, by Theorem \ref{strong_invariant}, $\Omega$
   must be strongly expansive.
 
   On the other hand suppose $\Omega$ is strongly expansive and has a
   totally disconnected transversal.  The result follows from Lemmas
   \ref{1D_cannot_expand} and \ref{strong-expansive-first-return}. 
 \end{proof}

A natural conjecture is that higher-dimensional tiling spaces are homeomorphic
to FLC tiling spaces if and only if they are strongly expansive and have 
totally disconnected transversals. We believe this conjecture to be false,
and we present the tiling of subsection
\ref{counter-ex} as a likely counterexample. However, if the geometry of 
the tiles can be controlled, then the result is true.

\begin{thm} \label{ILChomeoFLCcond}
Let $\Omega$ be a tiling space with totally disconnected 
transversal and with strongly expansive translational dynamics. If $\Omega$
is homeomorphic to a tiling space $\Omega'$ that has a finite number of
shapes and sizes of 2-tile patches (albeit possibly 
an infinite number of tile labels), then $\Omega$ is homeomorphic to
an FLC tiling space. 
\end{thm}

\begin{proof}
We will construct a homeomorphism from  $\Omega'$ (and hence from $\Omega$)
into a tiling space $\Omega''$ that is a suspension of a $\Z^d$ action on a 
totally disconnected space. 
Since strong expansivity is preserved by homeomorphisms, $\Omega''$ 
has strongly expansive dynamics, and hence has expansive dynamics. 
Standard results about Cantor maps then imply that $\Omega''$ is 
topologically conjugate to the suspension of a subshift, and hence to
an FLC tiling space. 

We first convert $\Omega'$ to a tiling space with a finite set of 
polyhedral tile shapes whose tiles meet full-face to full-face using Vorono\"{\i} cells as follows.
 Assign a point to the interior of each
prototile, such that prototiles of the same shape have the same marked point.
Instead of having a collection of labeled tiles, we then have a collection
of labeled points. Then associate each marked point to its Vorono\"{\i} cell ---
the set of points that are at least as close to that marked point as to
any other marked point.  This operation is a topological conjugacy. 

Next we apply the constructions of \cite{SW} to our tiling, noting
that the arguments of \cite{SW} only used the finiteness of the geometric
data associated with a tiling, and not the finiteness of the tile labels. 
We deform the sizes of each geometric class of tile such that the 
displacement between any two vertices is a vector with rational entries.
By rescaling, we can then assume that the relative position of any two
vertices is an integer vector. These shape deformations and rescalings are 
not topological conjugacies, but they do induce 
a homeomorphism of $\Omega'$ to the space $\Omega''$ of tilings by the 
deformed and rescaled tiles. 
There is then a natural map from $\Omega''$ to the $d$-torus,
associating to any tiling the coordinates of all of its vertices mod 1. 
Our $\R^d$ action on $\Omega''$ is then the suspension of a $\Z^d$ action 
on the fiber over any point of the torus. 
\end{proof}

\section{Invariant measures}\label{sec-measures}

The frequencies of patches in FLC tilings can be computed from translation-invariant Borel
probability measures.  There
are only countably many possible patches, and it is possible to assign
a nonnegative frequency to each one.  When the system has infinite
local complexity the space of patches can be uncountable. Frequencies
should then be viewed not as numbers but as measures on appropriate
spaces of patches.  We consider the general case first before
proceeding to fusion tilings.

\subsection{Measure and frequency for arbitrary ILC and FLC tilings.}
Consider any compact translation-invariant tiling space $\Omega$ with
a translation-invariant Borel probability measure $\mu$.  Whether or not
$\Omega$ has FLC, the measure
$\mu$ provides a reasonable notion of patch frequency that we now describe.
Let $\tilde \fff_n$ be the space of all connected $n$-tile
patches that appear in $\Omega$, modulo translation. $\tilde \fff_n$ 
inherits a $\sigma$-algebra of measurable sets from $\ppp_0$.
\footnote{A patch with
  $n$ tiles is described by specifying $n$ labels and $n$ locations,
  and each patch can be so described in $n!$ ways, one for each
  ordering of the tiles. Thus, $\tilde \fff_n$ is a subset of
  $[(\ppp_0)^n\times \R^{(n-1)d}]/S_n$, where $S_n$ is the group of
  permutations of the $n$ tiles.}

\begin{dfn}
Let $\tilde \fff_\infty = \cup_n \tilde \fff_n$ be the space of all
finite patches. A set $I \subset \fff_\infty$ of patches is called {\em
  trim} if, for some open set $U \subset \R^d$ and for every $\T \in
\Omega$, there is at most one pair $(P,\vecv) \in I \times U$ such
that $\T -\vecv$ contains the patch $P$. For each $U \subset \R^d$,
define the {\em cylinder set}
$$ 
\Omega_{I,U} = \{ \T \in \Omega \text{ such that there exists } 
(P,\vecv) \in I \times U  \text{ with } P \in \T-\vecv.\}
$$
The property of being trim, together with translation invariance,
implies that $\mu(\Omega_{I,U})$ is proportional to $Vol(U)$ for all
sufficiently small open sets $U$. The {\em abstract
  frequency} of a trim set $I$ of patches is
\begin{equation}
freq_\mu(I) = \mu(\Omega_{I,U}) / Vol(U),
\end{equation}
for any open set $U$ sufficiently small that each tiling is contained in
$\Omega_{I,U}$ in at most one way.  
\end{dfn}

This use of the word ``frequency" is justified by the ergodic theorem
as follows.  Let $\chi_{I,U}$ be the indicator function for
$\Omega_{I,U}$ and let $B_R(0)$ be the ball centered at
the origin of radius $R$. For $\mu$-almost every $\T
\in \Omega$ can define 
$$freq_{\mu,\T}(I) Vol(U) = \lim_{R \to \infty} \frac{1}{Vol(B_R(0))} \int_{\vecv \in B_R(0)} \chi_{I,U} 
(\T - \vecv) d\lambda,$$
where $\lambda$ is Lebesgue measure on $\R^d$.
Notice that the integral counts $Vol(U)$ each time an element of $I$ is in
$B_R(0) \cap \T$, up to small boundary effects.  Thus the
ergodic average is the average number of occurrences of $I$ per unit
area and represents a na\"{\i}ve  notion of frequency.  If $\mu$ is ergodic, then $freq_{\mu,\T}(I) = freq_\mu(I)$; if not, we
must integrate $freq_{\mu,\T}(I)$ over all $\T \in \Omega$ to get the true frequency $freq_\mu(I)$.

Since a subset of a trim set is trim, $freq_\mu$ can be applied to 
all (measurable) subsets of $I$, and is countably additive on such 
subsets. This follows from the countable additivity of $\mu$, and the 
fact that, for small $U$, 
the cylinder sets based on disjoint subsets of $I$ are disjoint. 
As a result, we can view $freq_\mu$ as a measure on any trim subset
of $\tilde \fff_\infty$. 

However, $freq_\mu$ should not be viewed as a measure on {\em all} of
$\tilde \fff_\infty$, since $\tilde \fff_\infty$ itself is not trim.  If
$I_1$ and $I_2$ are trim and disjoint but $I_1 \cup I_2$ is not
trim, then $\Omega_{I_1,U}$ and $\Omega_{I_2,U}$ are not disjoint
and in general $\mu(\Omega_{I_1 \cup I_2,U}) \ne \mu(\Omega_{I_1,U}) +
\mu(\Omega_{I_2,U})$. If we defined $freq_\mu(I_1 \cup I_2)$ to be
$\mu(\Omega_{I_1 \cup I_2, U})/Vol(U)$ anyway, then $freq_\mu$ would
not be additive.

\subsection{Measures for FLC fusion tilings}
We begin our investigation of invariant measures for fusion tilings by
reviewing how they work when the local complexity is finite.
In this case an individual patch $P$ will
usually have nonzero frequency and it suffices to look at
cylinder sets $\Omega_{P,U}$.  When the fusion is van Hove and
recognizable, it is possible to compute the frequency of an arbitrary patch 
from  the frequencies of high-order
supertiles, which depend primarily on the transition matrices of the fusion rule.
The reader can refer to \cite{Fusion} to flesh out the sketch we
provide here.

Let $j_n$ denote the number of distinct $n$-supertiles. For each
invariant measure $\mu$, let $\rho_n(i) \ge 0$ be the frequency of the
$i$-th $n$-supertile $P_n(i)$. Defining these frequencies requires
recognizability, so that we can uniquely determine whether a certain
supertile lies in a certain place in a tiling $\T$.  (Strictly
speaking, $\rho_n(i)$ is the sum of the frequencies of a family of
larger patches that consists of all possible extensions of the supertile out
to a specific ball that is larger than the recognizability radius.)

The vectors $\rho_n \in \R^{j_n}$ satisfy the 
{\em volume normalization} condition
\begin{equation}\label{old_volume_normalization}
\sum_{i=1}^{j_n}  \rho_n(i) Vol(P_n(i))
\end{equation}
and the {\em transition consistency} condition
\begin{equation}\label{old_transition_consistency} \rho_n = M_{n,N} \rho_N
\end{equation}
for each $N>n$.  For each patch $P$, 
$$ freq_\mu(P) = \lim_{n \to \infty} \sum_{i=1}^{j_n} \#(P \text{ in } P_n(i))
\rho_n(i),$$
where $\#(P \text{ in } P_n(i))$ denotes the number of patches equivalent to $P$
contained in the supertile $P_n(i)$. 

Since the supertile frequencies determine all the patch frequencies, and
since the patch frequencies are tantamount to the measures on cylinder
sets, an invariant measure on $\Omega$ is equivalent to a sequence of 
vectors $\rho_n$ that are non-negative, volume normalized, and 
transition consistent.    

It is possible to parameterize the space of invariant measures on
$\Omega_\rrr$ because of this result.  Transition consistency means
that $\rho_n$ is a non-negative linear combination of the columns of
$M_{n,N}$. Let $\Delta_{n,N} \subset \R^{j_n}$ be the convex polytope
of all such non-negative linear combinations of the columns of
$M_{n,N}$ that are volume normalized, and let $\Delta_n = \cap_N
\Delta_{n,N}$. The matrix $M_{n,m}$ maps $\Delta_m$ to $\Delta_n$.
The inverse limit $\Delta_\infty$ of these polytopes parametrizes the
invariant measures.  So we can see, for instance, whether or not a
certain fusion system is uniquely ergodic by looking at its transition
matrices.

\subsection{The frequency measure on $\ppp_n$ induced by $\mu$}
\label{frequencies_are_measures}
Now we begin to develop the parallel structures needed to handle the
situation where $\mu$ is a translation-invariant probability measure
on an ILC fusion tiling space.  Since the
space $\ppp_n$ may not be finite, instead of a vector $\rho_n \in
\R^{j_n}$, we have a measure on $\ppp_n$ that represents frequencies
of sets of $n$-supertiles.  Recall that $\lcal_n$, and hence $\ppp_n$,
comes equipped with a $\sigma$-algebra of measurable sets. 
As long as we have exercised reasonable
care in constructing $\ppp_n$ (e.g., if we have chosen our control 
points so that every element of
$\ppp$ contains an $\epsilon$-ball around the origin),
every subset $I\subset \ppp_n$ 
is automatically trim.
Thus for small enough $U$, $\mu(\Omega_{I,U})$ is proportional to
$Vol(U)$, and we define
$$ \rho_n(I) = \frac{\mu(\Omega_{I,U})}{Vol(U)}.$$
The non-negativity and (finite and countable) additivity properties of
$\rho_n$ follow from the corresponding properties of
$\mu$. Furthermore, we will show in Theorem
\ref{measures_are_sequences} that $\rho_n$ satisfies the volume
normalization condition
\begin{equation}
\int_{P \in \ppp_n} Vol(P)  d \rho_n = 1.
\end{equation}

\subsection{Three ways to view $M_{n,N}$}\label{transition.views}

An $n \times N$ matrix can be viewed as a collection of $nN$ numbers,
as an ordered list of $N$ vectors in $\R^n$, or as a linear transformation
from $\R^N$ to $\R^n$. Likewise, $M_{n,N}(P,Q)$ is a number, $M_{n,N}(*,Q)$
is a measure on $\ppp_n$, and $M_{n,N}$ is a linear map from measures on 
$\ppp_N$ to measures on $\ppp_n$. 

For a fixed $N$-supertile $Q$, we can view the $Q$th ``column''
$M_{n,N}(*,Q)$ of $M_{n,N}$ as a measure on $\ppp_n$, which we denote $\zeta_{n,Q}$. 
For any measurable $I \subset \ppp_n$, let
$$\zeta_{n,Q}(I) = \sum_{P \in I} M_{n,N}(P, Q) := \#(I \text{ in } Q),$$ 
the number of $n$-supertiles in $Q$ equivalent to those in $I$.
Although $I$ may be uncountable, there are only finitely many
$n$-supertiles in $Q$, so the sum is guaranteed to be finite. 
Likewise, for any measurable function $f$ on $\ppp_n$, and for fixed $Q$,  
\begin{equation}
\int_{P\in \ppp_n} f(P) d \zeta_{n,Q} = \sum_{P\in \ppp_n}  f(P) M_{n,N}(P,Q).
\label{integratetransition}
\end{equation}

As a linear transformation, $M_{n,N}$ maps measures on
$\ppp_N$ to measures on $\ppp_n$.  Consider a measure $\nu_N$ on $\ppp_N$. 
Define $M_{n,N}\nu_N$ on any measurable subset $I$ of $\ppp_n$ to be
\begin{equation}
 (M_{n,N} \nu_N)(I) = \int_{Q\in \ppp_N} M_{n,N}(I,Q) d \nu_N= \int_{Q \in \ppp_N} \sum_{P \in I}
M_{n,N}(P,Q) d\nu_N= \int_{Q\in \ppp_N} \zeta_{n,Q}(I) d \nu_N.
\label{measuretransitiondef}
\end{equation}
Every measure in the range of $M_{n,N}$ is a sum or integral over 
the frequency measures $\zeta_{n,Q}=M(*,Q)$ for different $Q$'s, just as every vector in the
range of a matrix is a linear combination of the columns of the matrix.

For brevity, we write $\nu_n$ for $M_{n,N} \nu_N$. 
Computing $\nu_n(I)$ involves  a sum over $P$ and an integral over $Q$.  Likewise, using $\nu_n$
to integrate functions over $\ppp_n$ also involves summing over $P$ and integrating over $Q$, with the 
sum inside the integral.
Specifically,
$$ \int_{P \in \ppp_n}  f(P) d\nu_n=  \int_{Q\in \ppp_N}  \sum_{P \in \ppp_n} f(P) M_{n,N}(P,Q)
 d\nu_N $$

It is straightforward to check that the composition of these
linear transformations is natural:  for $n<m<N$, $M_{n,N} \nu_N
= M_{n,m} (M_{m,N} \nu_N)$. 
Measures $\rho_n$ on $\ppp_n$ and $\rho_N$ on $\ppp_N$ are said to be
transition consistent if 
\begin{equation}\label{new_transition_consistency}
\rho_n = M_{n,N} \rho_N. 
\end{equation}
This has exactly the same form as the transition consistency condition
(\ref{old_transition_consistency}) for FLC fusions, only with the
right hand side now denoting the induced measure rather than simple
matrix multiplication.   We say that a sequence of 
measures $\{\rho_n\}_{n=0}^\infty$ is transition consistent if $\rho_n$ and
$\rho_N$ are transition consistent for all $n < N$.

Note that the measures $\zeta_{n,Q}$ are {\em not} volume normalized, since if $f(P) = Vol(P)$, then
$$\int_{P\in \ppp_n} Vol(P) d\zeta_{n,Q} = \sum_{P \in \ppp_N} Vol(P)
M_{n,N}(P,Q) = Vol(Q).$$  
However, if $\nu_N$ is volume normalized then
so is $\nu_n$, since
$$\int_{P \in \ppp_n} Vol(P) d\nu_n = 
\int_{Q\in \ppp_N}\sum_{P \in \ppp_n} Vol (P) M_{n,N}(P,Q) d\nu_N = \int_{Q \in \ppp_N} 
Vol(Q) d\nu_N = 1.$$

\subsection{Invariant measures for fusion tilings}

Specifying a measure for a tiling space is equivalent to specifying the 
measure for all cylinder sets, which in turn is equivalent to specifying
the (abstract) frequency for all trim families of patches.
For a large class of fusion rules, this can be
reduced to specifying a sequence of volume normalized and transition
consistent measures $\rho_n$ on $\ppp_n$:

\begin{thm}\label{measures_are_sequences}
  Let $\rrr$ be a 
  fusion rule that is van Hove and
  recognizable.  Each translation-invariant Borel probability measure
  $\mu$ on $\Omega_\rrr$ gives rise to a sequence of volume normalized
and transition consistent measures $\{\rho_n\}$ on $\ppp_n$.
Moreover, for any  trim set of patches $I$ 
\begin{equation} \label{measure_from_frequency}
freq_{\mu}(I) = \lim_{n \to \infty} \int_{P \in \ppp_n} \#( I \text{ in } P) d\rho_n,
\end{equation}
where $\# ( I \text{ in } P)$ denotes the number of translates 
of patches in the family $I$ that are subsets of $P$. 
Conversely, each sequence $\{\rho_n\}$ of volume normalized and transition consistent
measures 
corresponds to exactly one invariant measure $\mu$ via equation (\ref{measure_from_frequency}).  
\end{thm}

\begin{proof} Since volume normalization, transition consistency and 
equation (\ref{measure_from_frequency}) are linear conditions, 
and since all measures are limits of (finite) linear
combinations of ergodic measures, it is sufficient to prove these three
conditions for ergodic measures. 

Recall from Section \ref{frequencies_are_measures}
that $ \rho_n(I) = \frac{\mu(\Omega_{I,U})}{Vol(U)}$ is well-defined and
independent of our choice of (sufficiently small) $U$.  We will first prove 
that each $\rho_n$ is volume normalized, and then that the sequence is 
transition consistent.

We begin with $\rho_0$.
For any (measurable) set $I \subset \ppp_0$ of prototiles, let $U_I$ be
the intersection of the supports of the prototiles in $I$, let $V_I$ be
the union of the supports, and 
let $\Omega_I$ be the set of tilings where the origin is in a tile from $I$. 
Then $\Omega_{I,U_I} \subset \Omega_I \subset \Omega_{I,V_I}$, so 
$Vol(U_I) freq_\mu(I) = \mu(\Omega_{I,U_I}) \le \mu(\Omega_I) 
\le Vol(V_I) freq_\mu(I).$  
If all of the prototiles in $I$ had the same support, then 
$\Omega_{I,U_I}$ would equal $\Omega_I$, and, for any $P \in I$, we would have
$Vol(P)freq_\mu(I)=\mu(\Omega_I)$. 

Since the supports in $I$ are not
all the same, this equality does not hold exactly. However, for 
each $\epsilon>0$ we can find a $\delta>0$ such that, for all $I$ of
diameter less than $\delta$, $Vol(U_I)/Vol(V_I) > 1-\epsilon$. Furthermore,
for any prototile $P \in I$, we have $U_I \subset P \subset V_I$. This implies
that 
$$ \int_{P \in I} Vol(P) d\rho_0 \ge Vol(U_I) \rho_0(I) \ge (1-\epsilon) 
\mu(\Omega_{I,V_I}) \ge (1-\epsilon) \mu (\Omega_I), \hbox{ and }$$
$$\int_{P \in I} Vol(P) d\rho_0 \le Vol(V_I) \rho_0(I) \le (1-\epsilon)^{-1}
Vol(U_I) \rho_0(I) \le (1-\epsilon)^{-1} \mu(\Omega_I).$$

Now partition $\ppp_0$ into finitely many classes
$I_{1}, I_{2}, \ldots$ of Hausdorff diameter less than $\delta$. 
Then $1=\mu(\Omega) = \sum \mu(\Omega_{I_i})$, since the sets $\Omega_{I_i}$
overlap only on the set of measure zero where the origin is on the boundary
of two tiles. However, 
$$\int_{P \in \ppp_0} Vol(P) d\rho_0 = \sum_i \int_{P \in I_i} Vol(P) d\rho_0,$$
which is bounded between $(1-\epsilon)$ and $(1-\epsilon)^{-1}$. Since
$\epsilon$ is arbitrary, $\int_{\ppp_0} Vol(P) d\rho_0$ must equal 1, and 
$\rho_0$ is volume normalized. 
Exactly the same argument works
for $\rho_n$, using recognizability to replace $\Omega$ 
with $\omegan$. 

Next we prove transition consistency under the assumption that $\mu$
is ergodic.  From the ergodic theorem, there is a tiling $\T$ such that
spatial averages over the orbit of $\T$ can be used to compute the integral
of {\em every} measurable function on $\Omega$. In particular, 
for any measurable $I \subset \ppp_n$, we can compute the number of 
occurrences of $n$-supertiles of type $P \in I$ in a ball $B_R(0)$ 
of radius $R$ around the origin, sum over $P$, divide by $Vol(B_R(0))$,
and take a limit as $R \to \infty$. This limit must equal $\rho_n(I)$.   
In fact it is possible to show that for any measurable function $f$ on $\ppp_N$,
\begin{equation}
\lim_{R \to \infty} \frac{1}{Vol(B_R(0))} \sum_{Q \in \ppp_N} f(Q) \, \#(Q \text{ in } B_R(0)) \, = \,
\int_{Q\in \ppp_N} f(Q) d\rho_N,
\label{integratebyrhoN}
\end{equation}
where we let $\#(P \text{ in } B_R(0))$
be the number of occurrences of the $n$-supertile $P$ in the ball of radius 
$R$ around the origin in $\T$, counting only copies of $P$ that are 
completely in the ball. We then have 
$$
\#(P \text{ in } B_R(0)) \approx \sum_{Q \in \ppp_N} M_{n,N}(P,Q)\#(Q \text{ in }
B_R(0)),
$$ 
with the error coming from occurrences of $P$ in $N$-supertiles that are only
partially in $B_R(0)$, an error that is negligible in the $R \to \infty$
limit. Summing over $P \in I_n$, dividing by the volume of the ball, and 
taking the $R \to \infty$ limit, the left hand side becomes 
$\rho_n(I_n)=freq_\mu(I_n)$.  Letting $f(Q)$ be the function $M_{n,N}(I_n,Q)$
in equation (\ref{integratebyrhoN}) and using equation (\ref{measuretransitiondef}),
 the right-hand side becomes
$ (M_{n,N}\rho_N)(I_n)$.  In other words, $\rho_n$ and $\rho_N$
are transition consistent. 

We now turn to equation (\ref{measure_from_frequency}). 
Let $I$ be a trim set
of patches. Let $I^r \subset I$ be those patches whose diameter is at
most $r$. Restricting to integer values of $r$, we have that $I = \cup_r
I^r$, and hence that $freq_\mu(I) = \lim_{r \to \infty} freq_\mu(I^r)$. 
As with supertiles, 
$ freq_\mu(I^r) = \lim_{R \to \infty} \#( I^r \text{ in } B_R(0))/
Vol(B_R(0)).$
However, 
\begin{eqnarray*} \#( I^r \text{ in } B_R(0)) & = & 
\sum_{Q \in \ppp_N} \#(I^r \text{ in }Q) \#(Q \text{ in }B_R(0)) \cr 
&& + \#(I^r \text{ that spread over 2 or more $N$-supertiles
in }B_R(0)) \cr && + \#(I^r \text{ that intersect an $N$-supertile that is only 
partially in }B_R(0)).
\end{eqnarray*}
The contribution of the last term goes to zero as $R\to \infty$, since the
fractional area of the region within one $N$-supertile's diameter 
of the boundary 
of the ball goes to zero as $R\to \infty$. We then take the limit as 
$N \to \infty$, which eliminates the second term, since the fusion is 
van Hove and the patches in $I^r$ have bounded diameter. Using equation
(\ref{integratebyrhoN}) gives the
formula
\begin{equation} \label{rfinite}
freq_\mu(I^r) = \lim_{N \to \infty} \int_{Q \in \ppp_N} \#(I^r \text{ in }
Q) d\rho_N.
\end{equation}
Taking a limit of equation (\ref{rfinite}) as $r \to \infty$, 
and interchanging the order of 
the $r \to \infty$ and $N \to \infty$ limits, 
gives (\ref{measure_from_frequency}). This interchange is justified by the
fact that the integral on the 
right-hand-side of (\ref{rfinite}) is an increasing function
of both $r$ and $N$, so $\lim_{r\to\infty} \lim_{N\to \infty} = \sup_r \sup_N = 
\sup_{r,N} = \sup_N \sup_r = \lim_{N\to \infty} \lim_{r \to \infty}$. 
This completes the proof that a measure $\mu$ induces a
volume normalized and transition consistent sequence $\{\rho_n\}$ of 
measures on $\ppp_n$ satisfying equation (\ref{measure_from_frequency}).

Conversely, equation (\ref{measure_from_frequency}) gives frequencies
of trim families, and hence measures on cylinder sets, in terms of
the measures $\rho_n$.  The axiomatic properties of these measures
(e.g., additivity) follow directly from analogous properties of
frequencies, exactly as for FLC fusions (see \cite{Fusion} for
details).
\end{proof}  

\begin{cor} If $\rrr$ is a recognizable and van Hove
fusion rule, then the patches that are admitted in the limit have frequency
zero. 
\end{cor}

\begin{proof} In equation (\ref{measure_from_frequency}), the right
hand side is identically zero for any trim collection of patches 
that are admitted in the limit. 
\end{proof}
 
\begin{thm} If $\rrr$ is a recognizable and van Hove
fusion rule, and if each set $\ppp_n$ is finite, then frequency is
atomic. That is, for any trim set $I$ of patches, $freq_\mu(I)
= \sum_{P \in I} freq_\mu(P)$. 
\end{thm}

\begin{proof} Without loss of generality, we can assume that all patches
in $I$ are literally admitted, since all other patches have frequency zero.
However, since there are only finitely many supertiles at any level and 
finitely many patches in each supertile, 
there are only countably many literally admitted patches. Since 
$I$ is countable, the frequency of $I$ is the sum of the frequencies of its
elements. 
\end{proof}

\subsection{Parameterization of invariant measures for fusion tilings}

We next parametrize the space of invariant measures in terms
of the transition matrices $M_{n,N}$.  The construction is entirely analogous
to the parametrization of invariant measures for FLC tilings, only with
measures on $\ppp_n$ instead of vectors in $\R^{j_n}$. 

Let $\mmmn$ denote the space of 
volume normalized measures on $\ppp_n$. As noted above, 
$M_{n,N}$ maps $\mmmN$ to $\mmmn$. Let $\Delta_{n,N}= M_{n,N}\mmmN$.
(In the FLC case, this is the set of normalized
non-negative linear combinations of the columns of $M_{n,N}$.) 
Note that for $n<m<N$, $\Delta_{n,N} = M_{n,m} \Delta_{m,N} \subset \Delta_{n,m}$.
If we have
a sequence of volume-normalized measures $\{\rho_n\}$, then $\rho_n \in
\Delta_{n,N}$ for every $N$, and we define $\Delta_{n} = \cap_N \Delta_{n,N}$. 

Note that $M_{n,m}$ maps $\Delta_m$ onto $\Delta_n$. Define $\Delta_\infty$
to be the inverse limit of the spaces $\Delta_n$ under these maps. By 
definition, a point in $\Delta_\infty$ is a transition consistent 
sequence $\{\rho_n\}$ of volume normalized measures on $\{\ppp_n\}$. 

By Theorem \ref{measures_are_sequences}, $\Delta_\infty$ is the space of 
all invariant probability measures on $\Omega$.  We emphasize the following important corollary.

\begin{thm} \label{unique.ergodicity}
$\Omega$ is uniquely ergodic
if and only if each $\Delta_n$ is a single point. 
\end{thm}

\subsection{Measures arising from sequences of supertiles}

It is often possible to describe invariant measures in terms
of increasing sequences of supertiles. For each $N$-supertile $Q$, 
let $\delta_Q$ be the volume-normalized 
measure on $\ppp_N$ that assigns weight $Vol(Q)^{-1}$
to $Q$ and zero to all other $N$-supertiles. This induces measures 
$M_{n,N} \delta_Q$ on all $\ppp_n$ with $n<N$. If $I_n \subset \ppp_n$, then
$(M_{n,N}\delta_Q)(I_n)=\#(I_n \text{ in } Q)/Vol(Q)$.
In other words, $M_{n,N}\delta_Q$ describes
the frequency (number per unit area) of $n$-supertiles in $Q$. 
 
Let $\mu$ be an invariant measure on $\Omega$ and let $\{\rho_n\}$ be
the sequence of measures on $\ppp_n$ induced from $\mu$.  We say that
$\mu$ is {\em supertile generated} if there is a sequence of 
supertiles $\{Q_N\}$, with each $Q_N \in \ppp_N$ and $N$ ranging from 1 to
$\infty$, such that, for every $n$,
$\rho_n = \lim_{N \to \infty} M_{n,N} \delta_{Q_N}$. 

Supertile generated measures need not be ergodic (see \cite{Fusion} for a 
counterexample), but all ergodic tiling measures known to the authors
are supertile generated.
In most tiling spaces of interest, sequences of supertiles provide a 
useful means of visualizing the ergodic measures. 

\section{Complexity}
\label{sec-complexity}

An extremely powerful concept in the study of 1-dimensional subshifts
is that of complexity. For each natural number $n$, let the
{\em combinatorial complexity} $c(n)$ be
the number of possible words of length $n$. Usually this is done by
examining words within a fixed infinite (or bi-infinite) sequence, but
one can equally well consider words within any sequence in a subshift.
Many results about combinatorial complexity are known, such as:
\begin{itemize}
\begin{item}
If $c(n)$ is bounded, then all sequences are eventually periodic.
\end{item}\begin{item}
If $c(n)=n+1$ for all $n$ and the sequences are not eventually 
periodic, then each sequence is Sturmian and the space
of sequences is minimal.
\end{item}\begin{item}
If the sequences are non-periodic and 
come from a substitution, then $c(n)$ is bounded above and below by a 
constant times $n$. 
\end{item}\begin{item}
The topological entropy is $\limsup \log(c(n))/n$. 
\end{item}\end{itemize}
Likewise, for $\Z^2$ (or higher dimensional) subshifts we can count
the number of $n\times n$ square patterns, or $n \times m$ rectangular
patterns, but far less is known about higher dimensional combinatorial
complexity.

At first glance, computing the complexity of a tiling with infinite
local complexity would seem absurd. However, by adapting a
construction from studies of the topological pressure and topological
entropy of flows, we can formulate a notion of complexity that applies
to tilings, both FLC and ILC.

Let $\Omega$ be a space of $p$-dimensional tilings (not necessarily a
fusion tiling space), equipped with a metric on the space of
prototiles and hence a metric $d$ on the tiling space.  For each
$L>0$, let
$$d_L(\T_1, \T_2) = \sup_{x \in [0,L]^p} d(\T_1-x, \T_2-x).$$
That is, two tilings $\T_1$ and $\T_2$ are within $d_L$ distance $\epsilon$
if they agree on the region $[\epsilon^{-1}, L+\epsilon^{-1}]^p$, up
to $\epsilon$ changes in the labels, shapes, or locations of the tiles. A collection
$X$ of points in $\Omega$ is said to be {\em $(d_L,\epsilon)$-separated}
if the $d_L$ distance between distinct points in $X$ is bounded below 
by $\epsilon$.

\begin{dfn}The {\em tiling complexity function} $C(\epsilon,L)$ 
of a tiling space $\Omega$ is the maximum cardinality of a
$(d_L,\epsilon)$-separated set. 
\end{dfn}

$C(\epsilon,L)$ is closely related to the number of balls of $d_L$ radius
$\epsilon$ or $\epsilon/2$ needed to cover $X$. Specifically, the minimum
number of $\epsilon$-balls in an open cover is at most $C(\epsilon,L)$ and the minimum
number of $\epsilon/2$-balls is at least $C(\epsilon,L)$. 

In a 1-dimensional FLC tiling space where all tiles have length 1,
$C(\epsilon,L)$ is approximately $\epsilon^{-1}c([L + 2
\epsilon^{-1}])$, where $c$ is the combinatorial complexity of the
underlying subshift, since there are $c([L+2\epsilon^{-1}])$ choices for the
sequence of tiles appearing in $[-\epsilon^{-1}, L+\epsilon^{-1}]$
and $\epsilon^{-1}$ choices for where the origin sits within a
tile. In such examples, tiling complexity carries essentially the same
information as combinatorial complexity.

\begin{dfn} A tiling system has \begin{itemize}
\begin{item}
{\em bounded complexity} if
  $C(\epsilon,L)$ is bounded by a function of $\epsilon$, independent
  of $L$. 
\end{item}\begin{item}{\em polynomial complexity} if
  $C(\epsilon,L) < f(\epsilon) (1+L)^\gamma$ for some exponent
  $\gamma$ and some function $f(\epsilon)$. 
\end{item}\begin{item}
{\em$\epsilon$-entropy}  equal to $\limsup_{L \to \infty}
  (\log(C(\epsilon,L))/L^p)$. A priori this is non-decreasing as
  $\epsilon \to 0$. 
\end{item}\begin{item}
{\em finite entropy} if the
  limit of the $\epsilon$-entropy as $\epsilon \to 0$ is finite. The
entropy of the tiling space is defined to be that limit. 
\end{item}\end{itemize}
\end{dfn}
\noindent Note that entropy is not a pure number, but comes in units of (Volume)${}^{-1}$
since it describes the log of the number of patches of a given size per unit volume.

This next theorem implies that the way that $C(\epsilon,L)$ scales with $L$ is
preserved by topological conjugacy.

\begin{thm} Let $\Omega$ and $\Omega'$ be topologically
conjugate tiling spaces with metrics $d$ and $d'$ and 
tiling complexity functions $C$ and $C'$. Then, for
every $\epsilon>0$ there exists an $\epsilon'>0$ such that,
for every $L$, $C(\epsilon,L) \le C'(\epsilon',L)$.
\end{thm}

\begin{proof} Let $f: \Omega \to \Omega'$ be the topological
conjugacy. Since $f^{-1}$ is uniformly continuous, there
exists an $\epsilon'$ such that $d'(f(\T_1), f(\T_2))<\epsilon'$
implies $d(\T_1,\T_2)<\epsilon$. Thus, every $(d,\epsilon)$-separated
set in $\Omega$ maps to a $(d',\epsilon')$-separated set in
$\Omega'$. Since $f$ commutes with translation, $f$ also maps
$(d_L,\epsilon)$-separated sets to $(d'_L,\epsilon')$-separated sets
of the same cardinality.
Thus $C'(\epsilon',L)$ is bounded below by $C(\epsilon,L)$.
\end{proof}

\begin{cor} Suppose that the function $C(\epsilon,L)$ exhibits a property
that applies for all sufficiently small $\epsilon$. (For instance, that $C$ is
bounded in $L$, or has polynomial growth with a particular $\gamma$, or
has a particular entropy.) Then, for all sufficiently small $\epsilon$,
$C'$ has the same property.
\end{cor}
 
By contrast, the way that $C(\epsilon, L)$ scales with $\epsilon$ is {\em not}
topological.  It is even possible to have an FLC tiling space that is topologically
conjugate to an ILC space (see \cite{Radin-Sadun}, or subsection \ref{T2}, below). 

\begin{dfn} We say that the complexity function $C(\epsilon,L)$
{\em goes as} a given function $f(\epsilon,L)$ if $C(\epsilon,L)$ is
bounded both above and below by a constant times $f(\epsilon,L)$.
This is distinct from big-O notation, which only indicates an
upper bound.
\end{dfn}

\begin{ex}\label{solenoid}{\em Solenoids and random tilings.} 
  The dyadic solenoid, as described in Example \ref{sol-ex}, has
  bounded complexity. Recall that the tile set $\ppp_0$ is a 1-point
  compactification of a discrete set $A_0, A_1, \ldots$ labeled by 
the non-negative integers. For any $\epsilon$, let
  $N(\epsilon)$ be an integer such that the diameter of
  $\{A_{N(\epsilon)}, A_{N(\epsilon)+1}, \ldots, A_\infty\} \subset
  \ppp_0$ is less than $\epsilon$.  Then, for purposes of tiling
  complexity, the solenoid is essentially periodic with period
  $2^{N(\epsilon)}$, so the tiling complexity is bounded by
  $2^{N(\epsilon)}/\epsilon$, regardless of how big $L$ is.  The
  precise way that the tiling complexity scales with $\epsilon$ is
  somewhat arbitrary, depending on our choice of metric on $\ppp$, and
  hence on how $N(\epsilon)$ depends on $\epsilon$.

  Now consider the space of {\em all} tilings by the tiles $A_0, A_1,
  A_2, \ldots, A_\infty$ of the solenoid.  For each $\epsilon$, let
  $N'(\epsilon)$ be the cardinality of a maximal $\epsilon$-separated
  subset of $\ppp_0$. The tiling complexity $C(\epsilon,L)$ then goes as
  $N'(\epsilon)^{L+\frac{2}{\epsilon}}/\epsilon$, and the
  $\epsilon$-entropy is $\log(N'(\epsilon))$. For each $\epsilon$, the
  complexity is exponential in $L$, but the rate of exponential growth
  is different for different $\epsilon$. Since $N'(\epsilon) \to \infty$ 
  as $\epsilon \to 0$, this tiling space does not have finite
  entropy.
\end{ex}

The dyadic solenoid has bounded complexity, and the translational dynamics
are equicontinuous. This is not a coincidence. 

\begin{thm}If a tiling space has bounded complexity, then the
translational dynamics are equicontinuous. 
\end{thm}

\begin{proof} Let $\Omega$ be a tiling space with bounded complexity. 
We need to show that, for each $\epsilon>0$, there is a $\delta>0$ 
such that $d(\T, \T')< \delta$ implies that, for all $x \in \R^d$, 
$d(\T-x, \T'-x) < \epsilon$.

Let $d_\infty(\T, \T') = sup_x d(\T-x, \T'-x)$.  The fact that the
complexity $C(\epsilon/2, L)$ reaches a maximum of $N(\epsilon/2)$ for some 
$L$ and
then never grows implies that there is a $d_\infty$-separated set of
cardinality $N(\epsilon/2)$, but not one of cardinality $N(\epsilon/2)+1$. 
Let $\T_1, ... \T_{N(\epsilon/2)}$ be
such a maximal separated set. Since it is maximal, every other point
in the tiling space is within $d_\infty$-distance $\epsilon/2$ of one of
these points, so we can find an open cover of our tiling space with
balls of $d_\infty$-radius $\epsilon/2$ centered at each $\T_i$.

Now let $\delta$ be the Lebesgue number of this open cover with regard
to the original metric $d$.  That is, if any two tilings are within
$\delta$ of each other (in the $d$ metric), then there is an open set in
our cover that contains both of them.  But that means that both of
them are within $d_\infty$-distance $\epsilon/2$ of some $\T_i$, and hence
within $d_\infty$-distance $\epsilon$ of each other.  Which is to say,
their orbits always remain within $\epsilon$ of one another. \end{proof}

While the behavior of the complexity function is invariant under 
topological conjugacies, it is not preserved by homeomorphisms. The
stretched solenoid of Theorem \ref{can-expand} is homeomorphic to 
the dyadic solenoid, but its dynamics are not equicontinuous, and so its 
complexity function is not bounded.

\section{Geography of the ILC landscape}\label{sec-geography}

There are several ways to classify ILC tiling spaces. One is by how close the
dynamical and topological properties of the space come to those of an
FLC tiling space.

\begin{enumerate}
\begin{item} Is $\Omega$ an FLC tiling space?
\end{item}\begin{item}
If not, is it topologically conjugate to an FLC tiling space?
\end{item}\begin{item}
If not, is it homeomorphic to an FLC tiling space?
\end{item}\begin{item}
If not, does it have a totally disconnected transversal?
\end{item}\begin{item}
If not, does the transversal have finite topological dimension? 
\end{item}\end{enumerate}

The answer to any of these questions can be `yes' when the answers to the previous questions were `no'.  
In Section \ref{T2} we will exhibit an ILC tiling space that 
is topologically conjugate to an FLC tiling space. By applying the 
tile-stretching trick of Theorem \ref{can-expand} to that example, we
can construct a space that is homeomorphic to an FLC tiling space but that
is not topologically conjugate. The dyadic solenoid has a totally disconnected
transversal but is not homeomorphic to an FLC tiling space. The pinwheel
tiling has a 1-dimensional transversal. Finally, a space of tilings of the 
line by arbitrary intervals of length at least 1 and at most 2 has an 
infinite-dimensional transversal.   These assertions are discussed in detail in the examples of the appendix.

Another approach is to ask questions about the geometry and combinatorics 
of the tiles themselves:

\begin{enumerate}
\item Are there finitely many tile labels (and hence shapes and sizes)? 
\begin{enumerate}
\item[] If so, then infinite local complexity can only come from the
  ways that tiles slide past one another.  In two dimensions Kenyon
  showed \cite{Kenyon.rigid} that infinite local complexity arises
  only along arbitrarily long line segments of tile edges (fault
  lines), or along a complete circle of tile edges (fault circles).
  The common situation is to see shears along fault lines as in Example 2.
  We say that tilings of this type have {\em shear} infinite local 
complexity. They are discussed in \cite{Danzer, Kenyon.rigid,
    Me.Robbie, Me.Lorenzo, Sadun.pin}.
\end{enumerate}

\item If not, are there finitely many tile shapes and sizes? 
\begin{enumerate}
\item[] In this situation, ILC arises automatically, due to the infinitude of labels.   However, these spaces can have additional ILC arising geometrically, for instance from fault lines or circles.
\end{enumerate}

\item
If so, do the tiles meet full-edge to full-edge (or full-face to full-face
in dimensions greater than 2)? 
\begin{enumerate}
\item[] If there are only finitely many tile shapes,
meeting full-edge to full-edge, then the infinite local complexity comes 
entirely from the infinitude of labels. Using the techniques of \cite{SW},
we can recast the tiling flow as the suspension of a $\Z^d$ action on a 
transversal. In this case we call the complexity {\em combinatorially infinite}.
\end{enumerate}
\item Are there finitely many two-tile patches if some rotations are allowed?
\begin{enumerate}
\item[]
The pinwheel tilings are an example of this phenomenon.  Although it and other tilings of this class can be handled with the methods of \cite{BG, ORS, inverse}, many 
questions remain. We call this type of local complexity {\em rotationally infinite}.   
\end{enumerate}

\item

If any of these answers are ``no'', is the tiling space topologically
conjugate (or homeomorphic) to another tiling space where the answer 
is ``yes''? 
\end{enumerate}

The answers to these questions can depend on how we present the tilings. 
If an ILC tiling has only finitely many tile labels, then by working with
collared tiles we get a space with infinitely many tile labels. It is 
always possible to get tiles to meet full-edge to full-edge by using the 
Vorono\"{\i} trick we used in the proof of Theorem \ref{ILChomeoFLCcond}:  convert the tiling the a point pattern, and then consider 
the Vorono\"{\i} cells of the resulting points. However, this will typically 
result in an infinite number of tile shapes. 
In 1 dimension, it is always 
possible to resize the tiles to have unit length, so every 1-dimensional
ILC tiling space is homeomorphic (but not necessarily topologically conjugate)
to a space with combinatorially infinite local complexity. This is essentially the same as
studying the first return map on the transversal. 


Shear, combinatorial or rotational infinite local complexity 
are very special conditions
but these are the forms of ILC about which most is known.  We expect there
to be a wide variety of ways for the complexity of a tiling space to
be infinite, each with potentially different properties that have yet
to be discovered.  An important question will be to understand how
these properties are influenced by the topology, geometry, and combinatorics of
the tiles, tilings, and tiling spaces.

\begin{appendix}
\section{Examples}
\label{sec-examples}

In this section we go into more detail about the measures and complexity
of several ILC fusion tiling spaces.   These include the pinwheel, shear,
and solenoid examples already introduced elsewhere in the paper and also
some new tiling spaces.

\subsection{Toeplitz flows}\label{T2}

Toeplitz flows are variations on the dyadic solenoid construction
 of Example \ref{sol-ex}. In all cases that we will study, 
$\ppp_0$ is a compactification
of the set $\{A_0, A_1, A_2, \ldots\}$, and the supertiles take the form
$P_{n+1}(k) = P_{n}(k)P_n(n)$, where $k$ is either an integer greater than $n$
or a limit point.  The tilings all have an $A_0$ in every other position,
an $A_1$ in every 4th position, and generally an $A_n$ in every $(2^{n+1})$-st
position, and may include a single instance of a limiting tile (such as
the dyadic solenoid's $A_\infty$). The structure of the tiling space depends
on how many limit points there are, and on which sequences of tiles converge
to each limit point. The resulting tiling spaces are all measurably conjugate,
and hence have the same spectral properties, but are typically not
homeomorphic. They can often be distinguished by cohomology and by
gap-labeling group, and sometimes by complexity. 

A particularly interesting example involves two limit points
$A_\infty$ and $A'_\infty$, where the prototiles $A_{2n}$ with even
labels converge to $A_\infty$ and the prototiles $A_{2n+1}$ with odd
labels converge to $A'_\infty$.  Call this ILC fusion tiling space
$\Omega_{2}$. 

\begin{thm} The space $\Omega_{2}$ is topologically conjugate to the
(FLC) tiling space $\Omega_{PD}$ obtained from the period-doubling
substitution $X \to YX$, $Y \to XX$. 
\end{thm}

\begin{proof} The map from $\Omega_{2}$ to $\Omega_{PD}$ just
  replaces each $A_{2n}$ or $A_\infty$ tile with an $X$ tile and
  replaces each $A_{2n+1}$ or $A'_\infty$ tile with a $Y$ tile. The
  inverse map $\Omega_{PD} \to \Omega_{2}$ is slightly more
  complicated. Start with a period-doubling tiling.  Pick an arbitrary
  $Y$ tile $t_0$, and give the label $A_0$ to all the tiles that are
  an odd distance from $t_0$.  Then pick a remaining $X$ tile $t_1$,
  and give the $A_1$ label to all tiles a distance $2 \pmod{4}$ from
  $t_1$.  Pick a remaining $Y$ tile $t_2$ and give the $A_2$ label to
  all tiles a distance $4 \pmod{8}$ from $t_2$, etc.  If a tile does not
  eventually get a finite label $A_k$ from this process, replace it
  with an $A_\infty$ if it is an $X$ and with an $A'_\infty$ if it is
  a $Y$.
\end{proof}

We compute the complexity $C_2(\epsilon, L)$ of $\Omega_{2}$ in two ways,
once from the conjugacy to $\Omega_{PD}$ and once from the definition.

Since $\Omega_{PD}$ is an FLC substitution tiling, its complexity
(which we denote $C_{PD}(\epsilon,L)$ ) must be linear in that it goes
as $L/\epsilon$ for $L> 1/\epsilon$.  As noted earlier, the dependence
on $\epsilon$ is not a topological invariant, but the dependence on
$L$ is. Thus, for all sufficiently small $\epsilon$, the complexity of
$\Omega_{2}$ must grow linearly with $L$ also.  However, this argument
does not indicate how $C_{2}(\epsilon,L)$ scales with $\epsilon$.

Next, consider $\Omega_{2}$ directly. Let $\epsilon_0=d(A_\infty,A'_\infty)$ in
the metric on $\ppp_0$. If $\epsilon>\epsilon_0$, there exists an $N$
such that all tiles $A_n$ with $n \ge N$ are within $\epsilon$ of each other,
and also within $\epsilon$ of $A_\infty$ and $A'_\infty$. To within $\epsilon$, 
every tiling is then periodic with period $2^N$, and so the complexity is
bounded.

However, when $\epsilon < \epsilon_0$, then there exists an $N$ such that
all $A_n$ with $n \ge N$ are either within $\epsilon$ of $A_\infty$ or within
$\epsilon$ of $A'_\infty$, but not both. Furthermore, for $N$ sufficiently 
large, it is precisely the even labels that are close to $A_\infty$ and the 
odd labels that are close to $A'_\infty$. For counting purposes, we can 
replace all even labels $2n \ge N$ with $X$ and odd labels $2n+1 \ge N$ 
with $Y$,
as well as replacing $A_\infty$ with $X$ and $A'_\infty$ with $Y$. The remaining
tiles $A_n$ with $n < N$ are arranged periodically (with period $2^{N}$)
and do not affect the complexity for $L> 2^N$. Thus for $\epsilon<\epsilon_0$
and $L> 2^{N(\epsilon)}$, $C_2(\epsilon,L)$ is {\em exactly} the same
as $C_{PD}(\epsilon,L)$, and goes as $L/\epsilon$. 

We can also consider Toeplitz flows with two limit points, only with the
limiting structure more complicated than ``$A_{2n} \to A_\infty$, $A_{2n+1} 
\to A'_\infty$.''  Partition the non-negative integers into two infinite
sets $S$ and $S'$, and suppose that $A_n$ with $n\in S$ converges to 
$A_{\infty}$ while $A_n$ with $n\in S'$ converges to $A'_{\infty}$. Let 
$\alpha = \sum_{n \in S}2^{-n}$. Since $S$ and $S'$ are both infinite, 
$\alpha$ is not a dyadic rational; all numbers between 0 and 1 that are
not dyadic rationals are possible values of $\alpha$. Call the resulting
tiling space $\Omega_\alpha$.  

All spaces $\Omega_\alpha$ are measurably conjugate to the dyadic solenoid. 
However, they are not homeomorphic to the dyadic solenoid, and are typically
not homeomorphic to each other. The first \v Cech cohomology of $\Omega_\alpha$
works out to be $\Z[1/2] \oplus \Z$ (as opposed to $\Z[1/2]$ for the 
solenoid). The gap-labeling group, which is the Abelian group generated by
the measures of the clopen subsets of the transversal, is $\Z[1/2] + \alpha \Z$.
If $\alpha$ and $\alpha'$ differ by a dyadic rational, then the gap labeling
groups of $\Omega_\alpha$ and $\Omega_{\alpha'}$ are the same, and in fact 
$\Omega_\alpha$ and $\Omega_{\alpha'}$ are topologically conjugate. If $\alpha
-\alpha'$ is not a dyadic rational, then the gap-labeling groups are different
and the spaces are not topologically conjugate. If furthermore $\alpha$ and
$\alpha'$ are not both rational, then it turns out that the spaces are
not even homeomorphic. 

We can even consider Toeplitz flows with an infinite limit set. For
instance, imagine that the limit set is a $k$-sphere, with the points 
$A_0, A_1, \ldots$ corresponding to a dense subset of that sphere. In that 
case, the limit structure is detected by higher cohomology groups, with
$\check H^{k+1} = \Z$.

\subsection{The pinwheel tiling}

The pinwheel tiling is the most famous example of a tiling with rotational
ILC. The fusion rules shown in Figure \ref{pinfusion} are rotationally
invariant.  The decomposition of a right-handed supertile involves
rotation by $\alpha = \arctan(1/2)$ plus multiples of $\pi/2$, while
decomposition of a left-handed supertile involves rotation by
$-\alpha$. Since each $N$-supertile consists of both right-handed and
left-handed $(N-1)$-supertiles, a right-handed $N$-supertile will
contain tiles that are rotated by $N\alpha$, $(N-2)\alpha$, \ldots,
$(2-N)\alpha$ relative to the supertile (plus multiples of
$\pi/2)$. In the $N \to \infty$ limit, the orientation and directions
of the tiles are uniformly distributed on $\Z_2 \times S^1$.
Likewise, for any fixed $n$ the distribution of $n$-supertiles is
uniform in the $N \to \infty$ limit.  This establishes the uniqueness
and rotational invariance of the translation-invariant measure on
$\Omega_{pin}$.  (For details of this argument, see
\cite{Radin.pinwheel}.) Specifically, $d\rho_n$ is $d\theta/[4\pi
\times 5^n]$ on each of the two circles in $\ppp_n$.

For small $\epsilon$, the complexity $C(\epsilon,L)$ of the pinwheel
tiling goes as $L^3/\epsilon^3$. To specify a patch of size $L$ to within
$\epsilon$, one must specify the type of supertile 
containing the patch, and then specify where in the supertile the
patch lies.  (If the patch straddles two supertiles, then one must
specify two supertiles, but for each supertile there are only a
bounded number of choices for the nearest neighbors, so this does not
affect the scaling.) 
The direction of the supertile must be
specified to within $\epsilon/L$, since a rotation by $\epsilon/L$ of
a patch of size $L$ will move some tiles a distance $\epsilon$. Thus
the number of possible supertiles goes as $L/\epsilon$, while the number
of possible positions in the supertile goes as $L^2/\epsilon^2$, for
a total complexity that goes as $L^3/\epsilon^3$.

This is in contrast to the situation for self-similar FLC fusion
tilings. In a self-similar FLC fusion, there are a bounded number of
$n$-supertiles and a bounded number of ways that two $n$-supertiles
can meet. To specify a patch to within $\epsilon$, one must choose
among the finitely many supertiles of a given size and pick a location
within that supertile, resulting in a complexity $C(\epsilon, L)$ that goes
as $L^2/\epsilon^2$.  As noted earlier, the difference between scaling
as $\epsilon^{-2}$ vs.~$\epsilon^{-3}$ is not
significant, but the difference between scaling as $L^2$
vs.~$L^3$ is topological.  This remark applies to our next three examples as well.

\subsection{The anti-pinwheel and two hybrids}\label{sec-hybrid}

The {\em anti-pinwheel} tilings \cite{Radin.pinwheel} 
have tiles and supertiles with the same
shape as pinwheel tiles and supertiles. The difference is that $n$-supertiles
are built from $(n-1)$-supertiles as shown in Figure \ref{anti.pin.fig}.
\begin{figure}[ht]
\includegraphics[width=1.5in]{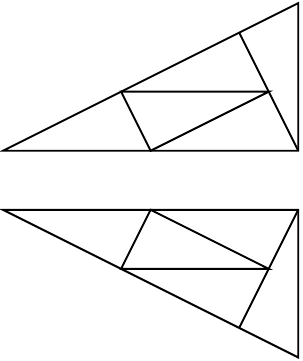}
\caption{The anti-pinwheel fusion rule}
\label{anti.pin.fig}
\end{figure}
An $n$-supertile of type $(R,\theta)$ is comprised of three $(n-1)$-supertiles
of type $(L,\theta+\alpha)$, one of type $(L,\theta+\alpha+\pi)$ and 
one of type $(L,\theta + \alpha + \pi/2)$. Likewise, an 
$n$-supertile of type $(L,\theta)$ is comprised of three $(n-1)$-supertiles
of type $(R,\theta-\alpha)$, one of type $(R,\theta-\alpha+\pi)$ and 
one of type $(R,\theta - \alpha - \pi/2)$.

Note that every daughter $(n-1)$ supertile has the same handedness and the
same direction (up to multiples of $\pi/2$), hence that every grand-daughter
$(n-2)$ supertile has the same handedness and direction (up to $\pi/2$), and
so on. If $N$ is even, then an $(R,\theta)$ $N$-supertile consists only of
tiles of type $(R,\theta)$, $(R,\theta+\pi/2)$, $(R,\theta+\pi)$ and
$(R,\theta-\pi/2)$. If $N$ is odd, the only tiles are of type 
$(L,\theta+\alpha)$, $(L,\theta+\alpha+\pi/2)$, $(L,\theta+\alpha+\pi)$, and
$(L,\theta+\alpha-\pi/2)$. Each anti-pinwheel tiling exhibits only four
of the uncountably many tile types!

The anti-pinwheel tiling space is neither minimal nor uniquely
ergodic. There is a minimal component for each handedness and each
direction (mod $\pi/2$). Each minimal component is uniquely ergodic,
with a measure for which $\rho_n$ is supported on a single handedness
and four angles, spaced $\pi/2$ apart. As with any self-similar FLC
tiling, each ergodic component has complexity that goes as $L^2/\epsilon^2$,
while the entire tiling space has complexity that goes as $L^3/\epsilon^3$, for
the same reasons as the pinwheel.

We next develop a hybrid between the pinwheel and 
anti-pinwheel.  This hybrid admits an action of the Euclidean group
and is minimal
with respect to translation. However, the system is not uniquely ergodic and the
ergodic measures are not rotationally invariant. 

The hybrid $n$-supertiles have the same supports as the pinwheel
(or anti-pinwheel) $n^2$-supertiles.  That is, they are 
1-2-$\sqrt{5}$ right triangles scaled up by $5^{n^2/2}$. 
Note that each $n$-supertile consists of $5^{2n-1}$ $(n-1)$-supertiles. 

The decomposition of $n$-supertiles into $(n-1)$-supertiles is as
follows.  First decompose the $n$-supertile into 5 smaller triangles
as with the anti-pinwheel. Repeat this process $2n-2$ times. Then pick
one of the $5^{2n-2}$ triangles and subdivide it using the pinwheel
pattern, while subdividing the other $5^{2n-2}-1$ triangles using the
anti-pinwheel pattern. The choice of which triangle to subdivide using
the pinwheel rule is arbitrary but has to be made consistently --- it
is part of our fusion rule.  For definiteness, we can choose to apply
the pinwheel rule to the triangle that contains the center-of-mass of
the $n$-supertile, as shown in Figure \ref{hybrid.pin.fig}.

An $n$-supertile of type $(R,\theta)$ will then consist of two 
right-handed $(n-1)$-supertiles and $5^{2n-1}-2$ left-handed $(n-1)$-supertiles,
all with angle $\theta+\alpha$ plus multiples of $\pi/2$. 
An $n$-supertile of type $(L,\theta)$ will consist of two 
left-handed $(n-1)$-supertiles and $5^{2n-1}-2$ right-handed $(n-1)$-supertiles,
all with angle $\theta-\alpha$ plus multiples of $\pi/2$. 
As a consequence, in a $2n$-supertile, a fraction greater than
$\prod_{k=2}^{n} \left ( 1 - \frac{2}{5^{2n-1}} \right ) > 0.98$ of the tiles
will have the same angle $\theta$ as the supertile (up to multiples of $\pi/2$), and 3/5 of those will
have the same handedness as the supertile. The distribution of angles
and handedness within a supertile is far from uniform, even in the 
$2n \to \infty$ limit. 

\begin{figure}\label{hybrid.pin.fig}
\includegraphics[width=6in]{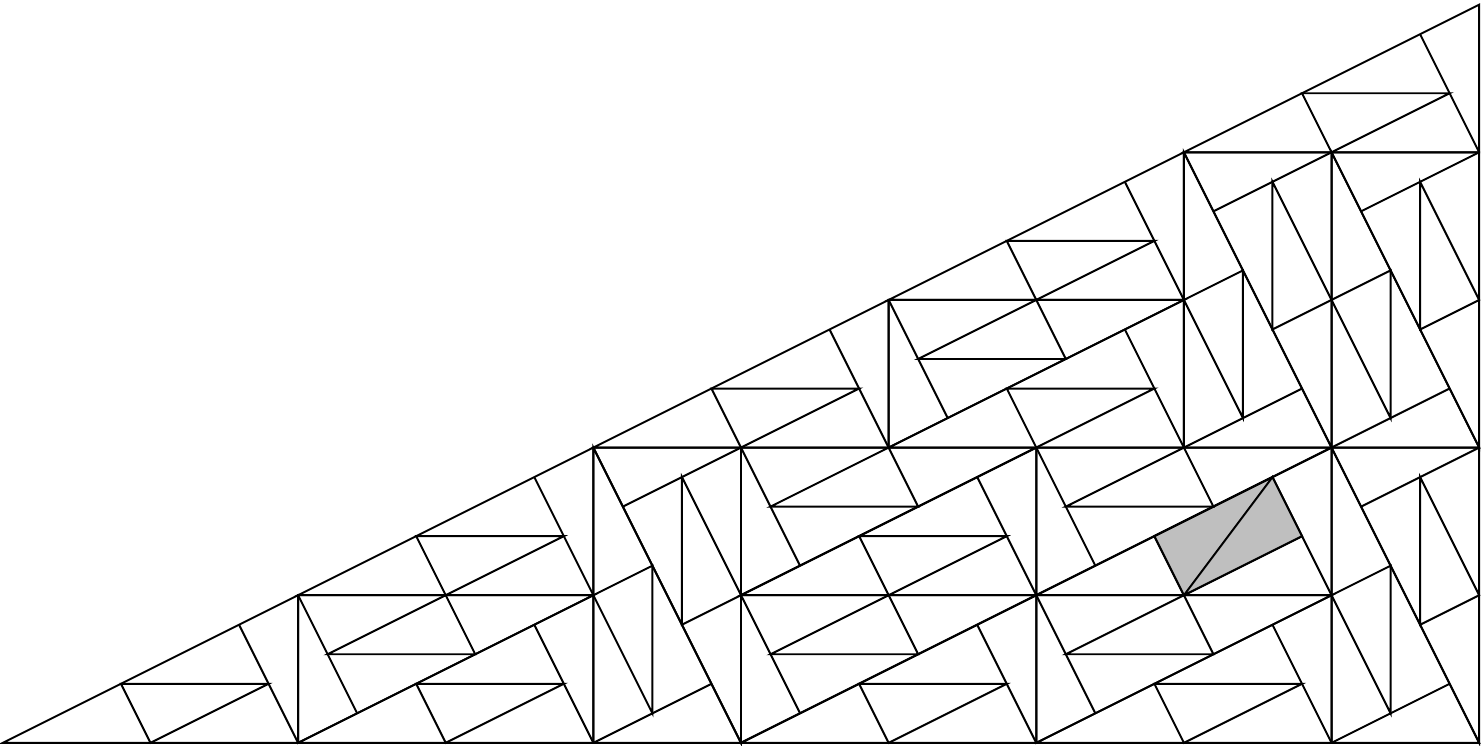}
\caption{A right-handed 2-supertile expressed as a union of 125 1-supertiles. 
Only two of the 1-supertiles (shaded) are right-handed.}
\end{figure}

For each $(H,\theta)$, there is a supertile generated measure coming
from the sequence of $N=2n$-supertiles of type $(H,\theta)$. Two such
measures, one with $(H,\theta)$ and one with $(H',\theta')$, agree
only if $H'=H$ and if $\theta-\theta'$ is a multiple of $\pi/2$. It is
not hard to see that these measures are extreme points of
$\poly_\infty$, insofar as they maximize the frequency of $(H,\theta)$
tiles, and hence are ergodic.

For these ergodic measures, $\rho_n$ is atomic. If $n$ is even, then
there is a large frequency of $(H,\theta+ m\pi/2)$ $n$-supertiles, a smaller 
frequency of $(H,\theta \pm 2\alpha + m \pi/2)$, or of supertiles of
the opposite handedness, a still smaller frequency with angle 
$\theta \pm 4\alpha \pm m\pi/2$, etc. There are countably many 
possibilities, but they occur with rapidly decreasing frequency. 

However, the fusion rule is primitive, and hence the tiling space
is minimal. 
Each $N$-supertile contains both
right and left-handed $n$-supertiles, each in $4(N-n-1)$ different 
directions. As $N \to \infty$, these directions become dense in $S^1$.
For any $n$-supertile $P$ and any $\epsilon>0$, one can therefore find 
an $N$ such that every $N$-supertile contains an $n$-supertile with
handedness and direction within $\epsilon$ of $P$. 

This hybrid, like the pinwheel and (complete) anti-pinwheel spaces, has
complexity that goes as $L^3/\epsilon^3$. As before, the number of 
possible directions of a supertile 
goes as $L/\epsilon$, while the number of locations
for the origin goes as $L^2/\epsilon^2$. 

\subsubsection{A hybrid with totally disconnected transversal and 
strongly expansive dynamics} \label{counter-ex}

Since the fusion rule for the hybrid pinwheel-anti-pinwheel 
is rotationally invariant, there is an $S^1$ action
on the transversal. In particular, the transversal is not disconnected. 
By modifying the construction to break this rotational invariance, 
we can get a tiling space that 
is minimal, has totally disconnected transversal, and is strongly expansive. 

Let $C$ be a Cantor set obtained by disconnecting the circle at the
countable set of points $n\alpha + m \pi/2 \pmod{2\pi}$, where $n \in
\Z$ and $\alpha = \arctan(1/2)$. That is, we remove each point $x$ of
the form $n \alpha + m \pi/2$ from the circle and replace it with 2
points, $x^+$ considered as the limit from above and $x^-$ considered
as a limit from below.  For definiteness, pick a metric on $C$.  There
is an obvious continuous map from $C$ to $S^1$, and we call elements
of $C$ ``angles'' and denote them $\theta$, with the understanding
that some angles require a superscript. Note that addition of $\alpha$
is well-defined on $C$, sending $\theta^\pm$ to $(\theta+\alpha)^\pm$.  
Likewise, subtraction of $\alpha$ and addition of multiples of
$\pi/2$ are well-defined.

Let $\ppp_n = \Z_2 \times C$, with the $n$-supertiles
having the same supports as with the hybrid pinwheel. To the two
points in $C$ with angle $\theta= k\alpha + m\pi/2$ we associate two
different (super)tiles. These (super)tiles
will have the same support but have different labels and may have different 
decompositions into lower-level supertiles. 

For each $n>1$, consider a partition of $C$ into $5^{2n-3}$ clopen sets,
such that the diameter of these sets (in the metric on $C$) goes to
zero uniformly as $n \to \infty$. To each of these clopen sets, associate
a triangle in the anti-pinwheel decomposition of the $n$-supertile into
$5^{2n-2}$ smaller triangles.

To decompose an $n$-supertile of type $(H,\theta)$ into
$(n-1)$-supertiles, first divide it into $5^{2n-2}$ smaller triangles
by applying the anti-pinwheel decomposition $2n-2$ times.  Then pick
the triangle associated to $\theta$, and decompose it into 5 triangles
using the pinwheel rule. Decompose the other $5^{2n-2}-1$ triangles
using the anti-pinwheel rule.

As a result of the $\theta$-dependence, this fusion rule breaks 
rotational symmetry and the rotation group
does not act on $\Omega$. For any distinct $\theta_{1,2} \in C$, 
there exists an $n$ such that $\theta_1$ and $\theta_2$ are not in the
same division of $C$ into $5^{2n-3}$ clopen sets. This implies that an
$n$-supertile of type $(H, \theta_2)$ is combinatorially different from
an $n$-supertile of type $(H,\theta_1)$.

This makes the dynamics strongly expansive and makes the transversal totally
disconnected.  Two points $\T_1$ and $\T_2$ in the same transversal
must have supertiles at some level that are different.  This means
that the combinatorics of the two tilings are different, so it is possible 
to find a clopen subset of the
transversal, defined using combinatorial data, that contains $\T_1$
but not $\T_2$. Furthermore, there is
a translation $x$ such that $\T_1$ and $\T_2$ differ by more than
$\epsilon$, which establishes expansivity. To see that the tiling space
is {\em strongly} expansive, note that the combinatorics of $\T_1$ and 
$\T_2$ are the same out to some distance and then are suddenly different; 
a time change could not account for the difference while preserving the 
match out to that point. 

This example is minimal but not uniquely ergodic, for the exact
same reasons that the hybrid pinwheel is minimal but not uniquely
ergodic. The ergodic measures are obtained from sequences of
$2n$-supertiles with fixed $(H,\theta)$, and the complexity goes as
$L^3/\epsilon^3$. The existence of uncountably many ergodic measures 
is evidence that this tiling space is not homeomorphic to an FLC tiling 
space.

\subsection{A direct product variation with shears} 

We return to the tiling of Example \ref{Non-Pisot_DPV} as an example of 
shear ILC.  Each label set
$\lcal_n$ consists of just four points, and the transition matrices
$M_{n,n+1}$ all equal $\left[\begin{matrix} 1 & 1 & 1 & 1\\ 3 & 0 & 3
    & 0\\3 & 3 & 0 & 0\\9 & 0 & 0 & 0 \end{matrix} \right]$. This
means that the transition-consistency requirements on the measures
$\rho_n$ are exactly the same as for an FLC substitution with the same
matrix. Each $\rho_n$ must be proportional to the Perron-Frobenius
eigenvector, $\left (\begin{smallmatrix} \lambda \cr 3 \cr 3 \cr -3
    \bar \lambda \end{smallmatrix} \right )$, where $\lambda =
(1+\sqrt{13})/2$ and $\bar \lambda = (1-\sqrt{13})/2$.  The measures
of all literally admitted patches can then be recovered from equation
(\ref{measure_from_frequency}), while the patches that appear in the
limit have total frequency zero.

The measures $\rho_n$ are independent of the side length $\alpha$
(except for overall normalization constants).   Perron-Frobenius theory
tells us that each $\Delta_n$ consists of a single element,
so by Theorem \ref{unique.ergodicity} this DPV must
be uniquely ergodic.

However, the set of
possible patches, and the count $\#(I \hbox{ in } P)$ on the right
hand side of equation (\ref{measure_from_frequency}) very much do
depend on $\alpha$.  
When $\alpha$ is rational, the tiling has FLC since if $\alpha =
p/q$ then the offsets between neighboring tiles must be multiples of
$1/q$.  In fact, all such multiples occur. Moreover, the offsets
between neighboring $n$-supertiles is also all multiples of $1/q$, up
to the size of the supertiles in question. Despite being FLC, the
complexity goes as $qL^3/\epsilon^2$, 
since the number of ways that two supertiles
of size $L$ can meet is $qL$. As long as $L > \epsilon^{-1}>q$,
$C(\epsilon,L)$ goes as $qL^3/\epsilon^2$.  

When $\alpha$ is irrational, then the tiling has ILC \cite{Me.Robbie},
and the offsets between neighboring supertiles is arbitrary.  A
countable and dense set of offsets is literally admitted, a
continuum of offsets appears in the limit, and the transversal
has topological dimension 1. To distinguish patches to
within $\epsilon$, we must specify the offset to within $\epsilon$,
and there are $L/\epsilon$ ways to do so.  $C(\epsilon,L)$ then goes
as $L^3/\epsilon^3$.  (This argument also applies to rational $\alpha$ when
$\epsilon > 1/q$. The distinction between rational and irrational $\alpha$
disappears when looking at structures of size greater than $1/q$.)

\subsection{A tiling of $\R$ with variable size tiles}

We next consider a 1-dimensional tiling whose tiles appear in a
continuum of sizes.  The possible lengths of $n$-supertiles are
$[{(3/2)^n}, {3(3/2)^n}]$, and the fusion rule making an $n$-supertile
depends on its length $x$.  If $x$ is above a certain threshold
(namely $2(3/2)^n$), then it is composed of two $(n-1)$ supertiles,
one of length $x/3$ and the other of length $2x/3$.  If $x$ is below
the threshold the $n$-supertile is composed of one $(n-1)$ supertile,
whose length is $x$, which we have ensured is an allowable $(n-1)$-supertile length.\footnote{Readers familiar with
  inflate-and-subdivide rules can think about this as an inflation by
  a factor of $3/2$ followed by a subdivision only if the tile length
  is above the threshold.} 
  
  We would like to allow the label sets of
$n$-supertiles to be their possible lengths, but doing so would make the fusion rule 
discontinuous.
To force continuity, we 
disconnect the line at each length $x$ of discontinuity,  making two
new points we call $x^+$ and $x^-$.  The discontinuities happen for
$x = 3^k(3/2)^m$, where $k$ and $m$ are nonnegative integers; we call
such values of $x$  {\em special}.  If $x$ is special, then for 
small $\epsilon > 0$
we say that $x-\epsilon$ is close to $x^-$ but not $x^+$, while $x + \epsilon$
is close to $x^+$ but not $x^-$.
Our label sets are intervals $\lcal_n = [{(3/2)^n}^+, {3(3/2)^n}^-]$ in
the disconnected line, each of which breaks up into $O(n)$ ordinary
closed intervals. (E.g., $[1^+,3^-] = [1^+,(3/2)^-] \cup [(3/2)^+,
(9/4)^-] \cup [(9/4)^+, 3^-]$.)  These labels describe the lengths of
the $n$-supertiles. 1-supertiles of label $3^+$ and $3^-$ each have
length 3, but have different decompositions into ordinary tiles.

To extend the fusion rule described above to our newly labeled
supertile sets, note that if $x$ is special, then either $x/3$ or
$2x/3$ is special (or both). In this case we use the
$(n-1)$-supertiles with the same superscript as $x$ itself.  Since
$2(3/2)^n = 3(3/2)^{n-1}$ is special, we have eliminated the
discontinuity in the rule.

The fusion rule is van Hove, recognizable, and primitive.    Thus the 
tiling dynamical
system is minimal and we may use Theorem \ref{measures_are_sequences} to 
determine the invariant measures.  In the next few sections we will show 
that the system is uniquely ergodic and compute the invariant measure.

\subsubsection{Transition-consistency}

\begin{lem} Every transition-consistent and volume-normalized 
family of measures $\{\rho_n\}$ is non-atomic.
\end{lem}

\begin{proof}
Consider an $n$-supertile of size $x$. We will show that, for $N$
large, dividing an $N$-supertile of size $y$ into $n$-supertiles
results in relatively few $n$-supertiles of label exactly $x$. Specifically,
we will argue that 
the fraction of the $n$-supertile descendants that are of size $x$ is
bounded above by a constant times $N^{-1/2}$, implying that no
transition-consistent measure can give weight greater than $cN^{-1/2}$
to $n$-supertiles of label $x$. Since $N$ can be chosen arbitrarily
large, the measure of the $n$-supertiles of label $x$ is zero.

If some of the descendants of $y$ have length $x$, 
then $y = x 3^k
(3/2)^m$ for some non-negative integers $k$ and $m$. Since $y/x$ is also 
approximately $(3/2)^{N-n}$, $N$ is bounded above and below by a constant times
$k+m$.  When
dividing $y$, 2/3 of the length winds up in a chunk of size $2y/3$ and
1/3 winds up in a chunk of size $y/3$. Dividing $k+m$ times, we get
size $x$ only if we go to the larger daughter $m$ times and the
smaller daughter $k$ times. There are at most ${k+m \choose m}$ ways
to do this, 
so the fraction of the length of $y$
consisting of $n$-supertiles of type $x$ is at most $\left(\frac{1}{3}
\right )^k \left ( \frac{2}{3}\right)^m {k+m \choose m}$. This is the
probability of getting $k$ heads and $m$ tails when flipping a biased coin
$k+m$ times. Since the variance in the number of heads is proportional to
the number of flips, and since the distribution is approximately normal
(by the central limit theorem), the probability of getting exactly $k$
heads is bounded by a constant times 
$(k+m)^{-1/2}$, and hence by a constant times $N^{-1/2}$.
\end{proof}

Since $\rho_n$ is non-atomic, we can treat $\rho_n$ as a continuous 
measure on the (ordinary) interval $[(3/2)^n, 3(3/2)^n]$. 
We write 
$$d \rho_n = f_n(x) dx,$$
to represent the frequency of $n$-supertiles whose lengths are between
$x$ and $x+dx$.  In principle, $\rho_n$ could have a singular
component, in which case $f_n$ should be understood as a distribution
rather than as an ordinary function.

Transition-consistency for the $\rho_n$s requires the $f_n$s to be
related, which is key to computing the ergodic measure.  By
definition, for any measurable $I \subset \ppp_n$ we need
$$\int_I f_n(x) dx = \int_{y \in \ppp_{n+1}} M_{n, n+1}(I, y) f_{n+1}(y) dy.$$   
To compute $M_{n,n+1}$ we identify three subsets of $\ppp_n$.  If $x
\in [{(3/2)^n}^+, {(3/2)^{n+1}}^-]$, then the $n$-supertile of label
$x$ must be the small daughter of an $(n+1)$-supertile of label $3x$.
An $n$-supertile with $x \in [{(3/2)^{n+1}}^+, {2(3/2)^n}^-]$ must be
the only child of an $(n+1)$ supertile of label $x$, and an
$n$-supertile of length $x \in [{2(3/2)^n}^+, {3(3/2)^n}^-]$ can
either be the only child of an $(n+1)$-supertile of label $x$ or the
large daughter of an $(n+1)$-supertile of label $3x/2$.  Thus, if $I
\subset [{(3/2)^n}^+, {(3/2)^{n+1}}^-]$ we see that $M_{n,n+1}(I,y) =
\chi_{3I}(y)$ and we have $$\int_I f_n(x)
dx 
= \int_{3I} f_{n+1}(y) dy = \int_I 3f_{n+1}(3x) dx,$$ implying that on
this interval we have $f_n(x) = 3f_{n+1}(3x)$.  The computation for
the other two intervals is similar and we obtain
$$f_n(x) = \begin{cases} 3 f_{n+1}(3x) & (3/2)^n \le x \le  {(3/2)^{n+1}}  \cr
f_{n+1}(x) & {(3/2)^{n+1}}  < x \le {2(3/2)^n} \cr
f_{n+1}(x) + (3/2)f_{n+1}(3x/2) & {2(3/2)^n < x \le 3(3/2)^n.}
\end{cases}
$$ 

Because the transition-consistency relationship does not depend on
$n$, we can renormalize all of the $f_n$'s and induce a map from
$L^2([1,3])$ to itself.  Define
$$ \tilde f_n(x) = (3/2)^{2n} f_n((3/2)^n x).$$ 
The quantity $\tilde f_n(x) dx$ represents $(3/2)^n$ times the
frequency of $n$-supertiles of length between $(3/2)^nx$ and $(3/2)^n
(x+dx)$, and $x$ ranges from 1 to 3. Equivalently, if we were to
rescale all lengths by $(3/2)^n$, it would represent the frequency
(number per {\em rescaled} length) of $n$-supertiles of length between
$x$ and $x+dx$.

The transition consistency equations become
$$\tilde f_n(x) = \begin{cases} (4/3) \tilde f_{n+1}(2x) & 1 \le x \le 3/2 \cr
(4/9) \tilde f_{n+1}(2x/3) & 3/2 < x \le 2 \cr
(4/9) \tilde f_{n+1}(2x/3) + (2/3) \tilde f_{n+1}(x) & 2< x \le 3, \end{cases}$$
or equivalently $\tilde f_n(x) = \ren \tilde f_{n+1}(x)$, where
$$\ren \tilde f (x) = \begin{cases} (4/3) \tilde f (2x) & 1 \le x \le 3/2 \cr
  (4/9) \tilde f(2x/3) & 3/2 < x \le 2 \cr (4/9) \tilde f(2x/3) +
  (2/3) \tilde f (x) & 2 < x \le 3. \end{cases}$$ In the next section
we see that $\ren$ is a diagonalizable operator on $L^2([1,3])$ and
compute its spectrum in order to show that there is a unique invariant
measure.
 
\subsubsection{The invariant measure}

\begin{thm} The tiling space is uniquely ergodic, and the unique measure has
\begin{equation}\label{answer}
\tilde f_n(x) = \begin{cases} \frac{c}{x^2} & x  \le 2 \cr \cr \frac{3c}{x^2} & x>2, 
\end{cases} \end{equation}
for every $n$, where $c = (3 \ln(3)- 2 \ln(2))^{-1}$.  
\end{thm}

\begin{proof} The proof follows the same general lines as the proof
of Theorem 8 of \cite{Sadun.pin}. 
To prove unique ergodicity, we show that the spectrum 
of $\ren$ is entirely in the unit disk, that 1 is the only eigenvalue on the
unit circle, and that there is a unique solution to $\ren \tilde f=\tilde f$, up to normalization,
namely the function listed above.

Suppose that $\tilde f$ is an eigenfunction of $\ren$ with eigenvalue $\lambda$. 
The equations $\ren \tilde f = \lambda \tilde f$ reduce to:
\begin{eqnarray*} &\tilde f(2x)/\tilde f(x) =3 \lambda / 4 & \qquad 1 \le x \le 3/2 \cr
&\tilde f(2x/3)/\tilde f(x) = 9 \lambda/4  & \qquad 3/2 < x \le 2 \cr
&\tilde f(2x/3)/\tilde f(x) =3(3\lambda-2)/4 & \qquad  2 < x \le 3
\end{eqnarray*}

Note that $\tilde f(x)$ determines $\tilde f(2x)$ or $\tilde f(2x/3)$, which determines
$\tilde f(4x/3)$ or $\tilde f(4x/9)$, which determines the next power of 2 times $x$,
divided by an appropriate power of 3.  At every stage, we add
$\ln(2)$ to $\ln(x)
\pmod{\ln(3)}$.  Since $\tilde f(x)$
determines a dense set of function values, there is at most one
function (up to scale) with eigenvalue
$\lambda$.  Irrational rotations yield a uniform measure on the 
circle, so we find
ourselves in the range [1,3/2] a fraction $\ln(3/2)/\ln(3)$ of the time,
in the range [3/2, 2] a fraction $\ln(4/3)/\ln(3)$ of the time and in the range [2,3] a fraction 
$\ln(3/2)/\ln(3)$ of the
time, regardless of the initial value of $x$. 
This means that we multiply by $3 \lambda/4$ an average of 
$\ln(3/2)$ times for 
every $\ln(4/3)$ times that we multiply by $9\lambda/4$ and for every
$\ln(3/2)$ times that we multiply by $3(3\lambda-2)/4$. Since the values of $\tilde f$ do
not grow or shrink in the long run, we must have
\begin{equation} \label{eig}
\left | \frac{3 \lambda}{4}\right|^{\ln(3/2)} \left | \frac{9 \lambda}{4}\right|^{\ln(4/3)} \left | 
\frac{3(3 \lambda-2)}{4}\right |^{\ln(3/2)} = 1,
\end{equation}
or equivalently 
\begin{equation}\label{alsoeig} \ln(2)\ln|\lambda| + \ln(3/2) \ln|3\lambda-2| = 0.
\end{equation}
If $|\lambda|>1$, then both terms on the left hand side of (\ref{alsoeig}) are positive. 
If $|\lambda|=1$ and $\lambda \ne 1$, the first term is zero and the second term is 
positive. Thus, the only possible eigenvalues are +1 and points
that are strictly inside the unit circle. 

It is straightforward to check that the function of equation (\ref{answer})
is in fact an eigenfunction of $\ren$ with eigenvalue 1.  The constant $c$ comes from
the volume normalization condition $\int_1^3 x\tilde f_n(x) dx = 1$. \end{proof}

Knowing the $\lambda=1$ eigenfunction is enough to determine the invariant measure. However, it is also
useful to identify the other eigenvalues and eigenfunctions, as these describe how averages over finite 
regions can differ from the the ergodic limit.   
If $\gamma$ is a complex number such that 
\begin{equation} \label{eig2} 
3^\gamma = 2^\gamma + 1,
\end{equation} 
then the function:
$$\tilde f(x) = \begin{cases} x^{-(\gamma+1)} &1 \le  x<2 \cr 3^\gamma x^{-(\gamma+1)} 
  & 2 < x \le 3 \end{cases}$$ 
is an eigenfunction of $\ren$ with eigenvalue
$\lambda = (3/2)^{\gamma-1}$.
There are countably many solutions to (\ref{eig2}), each corresponding 
to a solution to (\ref{eig}). (Condition (\ref{eig}) is necessary but 
not sufficient for $\lambda$ to be an eigenvalue.) 1 is an accumulation
point for the spectrum of $\ren$, so convergence under $\ren$ 
to the unique invariant measure is {\em not} exponential. 

\subsubsection{Complexity and transversal}

The complexity $C(\epsilon,L)$ goes as $L^2/\epsilon^2$.  To specify a
patch of size $L$ up to $\epsilon$ error, one must identify a
supertile of size bigger than $L$, specify the length of the supertile
to within $\epsilon$, and pick a location within that supertile to
within $\epsilon$, leaving $L^2/\epsilon^2$ possibilities.

Despite having a continuum of tile lengths, this tiling space has totally
disconnected transversal. The clopen sets describe the combinatorics of how
tiles fit into supertiles of various orders out to a certain distance from
the origin. To see this, imagine that tilings $T$ and $T'$ are in the
transversal and are close. This occurs when $T$ is, to a large distance,
a small dilation of $T'$ (or vice-versa). However, eventually $T'$ will
have a tile whose length is slightly less than $3$, while the corresponding
region of $T$, having length slightly greater than 3, consists of two tiles. 
Using this difference in combinatorics, we can construct a clopen set
that contains $T$ but not $T'$. 

\end{appendix}

\end{document}